\documentclass{article}
\usepackage[utf8]{inputenc}
\usepackage{amssymb, amsmath, amsfonts, blkarray, amsthm}
\usepackage[square]{natbib}
\usepackage{hyperref}
\usepackage{titlesec}
\usepackage{parskip}
\usepackage{listings} 
\usepackage{lineno} 
\usepackage[normalem]{ulem} 
\usepackage{xcolor, graphicx}
\usepackage{tabularx}
\usepackage{stmaryrd}
\usepackage{enumerate}
\usepackage{mathrsfs}
\usepackage{subcaption}
\usepackage{float}
\usepackage{booktabs}
\usepackage{authblk}

\newcommand{\KL}[1]{D_{\mathrm{KL}}^{#1}}

\newtheorem{theorem}{Theorem}[section]
\newtheorem{corollary}[theorem]{Corollary}
\newtheorem{lemma}[theorem]{Lemma}
\newtheorem{prop}[theorem]{Proposition}

\allowdisplaybreaks

\DeclareMathOperator{\Tr}{Tr}
\DeclareMathOperator*{\argmin}{arg\,min}
\DeclareMathOperator*{\argmax}{arg\,max}

\title{On additive averaging kernels for finite Markov chains}
\author[1]{Ryan J.Y. Lim\thanks{Email: ryan.limjy@u.nus.edu}}

\author[1]{Michael C.H. Choi\thanks{Email: mchchoi@nus.edu.sg, corresponding author}}

\affil[1]{Department of Statistics and Data Science, National University of Singapore, Level 7, 6 Science Drive 2, 117546, Singapore}
\date{\today}

\begin{document}

\maketitle
\begin{abstract}
    We study additive mixtures of Markov kernels of the form $A_\alpha = \alpha P + (1-\alpha)G$, where $\alpha \in [0,1]$, $P$ is a baseline sampler and $G$ is a Gibbs kernel induced by a partition of the state space. We first motivate the study of $A_\alpha$, which can be interpreted as the projection of a lifted Markov chain. We then consider the minimisation of distance to stationarity under two objectives: the squared Frobenius norm and the Kullback–Leibler (KL) divergence. For the Frobenius objective, we derive explicit trace formulae and identify a Cheeger-type functional that characterises optimal two-block partitions. This yields a structured combinatorial optimisation problem admitting a difference-of-submodular decomposition, enabling efficient approximation via majorisation–minimisation. We also obtain geometric decay rates governed by the absolute spectral gap of $P$. For the KL divergence, we establish convexity-based bounds showing that the divergence of $A_\alpha$ is controlled by those of both $P$ and $G$, thereby reducing partition selection to the Gibbs component. Numerical experiments on the Curie-Weiss model demonstrate that suitable choice of both the partition and the parameter $\alpha$ can significantly accelerate convergence in total variation distance. We observe a consistent trade-off between local exploration and global averaging, with intermediate values of $\alpha$ achieving the best performance across regimes. \\


\end{abstract}

\section{Introduction}
This work is motivated by recent advancements in group-averaged Markov chains, where \cite{choi2025groupaveragedmarkovchainsmixing,GPG} showed that, given a Gibbs kernel $G$, kernels of the form $GPG$ can exhibit significant improvements over the base kernel $P$. Improvements in metrics such as mixing time, spectral gap, and Kullback-Leibler (KL) divergence have been established. While such composition-based constructions are effective, they involve the composition of several kernels at each transition. This naturally motivates the study of additive analogues, which are structurally simpler.

In this paper, we study additive mixtures of the form $A_\alpha = \alpha P + (1-\alpha)G$, where $P$ is a given $\pi$-stationary Markov kernel and $G$ is the Gibbs kernel induced by a partition $\mathcal{X} = \bigsqcup_{i=1}^k\mathcal{O}_i$. Such additive constructions arise naturally as convex combinations of transition kernels and can be interpreted as randomised schemes that alternate between local updates via $P$ and global averaging within blocks via $G$. 

However, unlike the composition-based kernels previously studied, we do not claim a universal improvement over $P$. The performance and practicality of the additive mixture, as we will show, are highly dependent on the partition inducing $G$, as well as the mixing parameter $\alpha$. A poorly chosen partition may average over regions that are not useful for improving convergence, whereas a well-chosen partition can exploit meaningful structure in the state space.

We emphasise that exact implementation of $G$ requires the ability to sample from $\pi$ restricted to the blocks of the partition. Thus, the additive construction is not necessarily computationally cheap. There is a natural trade-off: coarser partitions may provide stronger averaging but can be more expensive to implement, while finer partitions are typically cheaper but may provide less global averaging.

From an algorithmic point of view, the partition $(\mathcal{O}_i)_{i=1}^k$ may be viewed as a user-chosen parameter. In many applications, such a partition can be induced by a natural summary statistic, for example magnetisation in spin systems, energy level, or simply an event and its complement. This motivates the study of optimal partitions under different objective functions, with the goal of studying how and why certain partitions perform better than others.

To achieve our aims, we formulate our optimisation problems based on the Frobenius norm and the KL divergence to stationarity. The Frobenius objective provides an averaged spectral measure of discrepancy and leads to explicit trace formulae, while the KL divergence gives an information-theoretic measure that is naturally suited to convexity-based bounds. These objectives allow us to study the effect of both the partition and the parameter $\alpha$ on convergence behaviour.

The numerical experiments then complement this theoretical perspective. Although the theory does not imply that $A_\alpha$ uniformly improves upon $P$, the experiments show that both natural and optimised partitions can lead to improved mixing when paired with a suitable choice of $\alpha$

We begin with a recap of key definitions in Section \ref{sec: prelim}, followed by the construction of a lifted kernel associated with $A_\alpha$ in Section \ref{sec: lift}. We show that the lifted kernel, when projected onto the first coordinate, recovers the mixture sampler $A_\alpha$, thereby providing an alternative interpretation of $A_\alpha$ as the marginal of a lifted Markov chain.

In Section \ref{sec: frobnorm}, we show that for the Frobenius objective, the optimisation over two-block partitions $\mathcal{X} = S \sqcup S'$ admits an explicit characterisation in terms of a Cheeger-type functional associated with the projection chain of $P$. We also present a singleton approximation as a computationally efficient alternative to full combinatorial optimisation. Furthermore, the optimisation problem admits a decomposition into a difference-of-submodular functions, enabling the use of approximation algorithms such as majorisation–minimisation methods to identify near-optimal partitions. Additionally, we derive spectral bounds that yield geometric decay rates for the Frobenius distance in terms of the absolute spectral gap of $P$.

In Section \ref{sec: KL}, we study the KL divergence and show that the additive structure yields convexity-based bounds in terms of the KL divergences of $P$ and $G$ to stationarity. As a result, the choice of partition influences the bound only through the contribution of $G$. We then explicitly characterise the optimal partition for this term, leading to a concrete upper bound on the KL divergence of $A_\alpha$ from $\Pi$.

In Section \ref{sec: exp}, we present numerical experiments using the Curie–Weiss model as a benchmark. Taking $P$ to be the Glauber dynamics, we compare the samplers $GPG$, $GP$, $A_{1/2}$, and $P$ under both fixed and optimised two-block partitions. We observe a consistent hierarchy in mixing performance across these samplers. We further investigate the role of the mixing parameter $\alpha$, which governs the trade-off between local exploration and global averaging. In contrast to the extreme cases $\alpha = 0$ and $\alpha = 1$, which correspond to purely averaging dynamics and purely local dynamics respectively, we demonstrate that intermediate values of $\alpha$ can significantly improve convergence.

\subsection{Related works}
A broad range of approaches have been proposed to accelerate the convergence of Markov chains beyond standard reversible samplers. \cite{chen_lifting} introduced lifting and state-space augmentation techniques, where auxiliary variables enlarge the state space to reduce diffusive behaviour and enable faster traversal across bottlenecks. Closely related are non-reversible samplers, which break detailed balance to induce persistent motion; \cite{TURITSYN2011410} show that such constructions can improve spectral gaps and asymptotic variance in various settings. These two directions have also been unified through a framework proposed in \cite{Vucelja_lifting}, in which non-reversible dynamics arise naturally from lifting constructions.

Another line of work considers mixtures of Markov kernels within the MCMC framework. Notably, the random scan Gibbs sampler which is widely used in practice (see, for example, \cite{levine_gibbs} for a review), can be formulated as a mixture of coordinate-wise Gibbs kernels. \cite{roberts_geometric} study such mixtures from the perspective of geometric convergence. However, in most of these approaches, the relative weights given to each kernel are fixed \emph{a priori}, and relatively little attention has been paid to how such combinations should be structured or optimised in relation to the geometry of the state space.

In our context, the additive kernel $A_\alpha$ studied in this paper can be interpreted as the projection of a lifted chain. Unlike typical lifting-based approaches, however, our construction preserves the reversibility of the base kernel $P$. Furthermore, the choice of the parameter $\alpha$, which determines the relative weighting between the baseline sampler $P$ and the Gibbs kernel $G$, plays a crucial role, as it has a significant impact on the performance of the mixture sampler $A_\alpha$, with intermediate values often yielding the best convergence behaviour.

\section{Preliminaries and setup} \label{sec: prelim}
For any integers $a \leq b \in \mathbb{Z}$, we set $\llbracket a,b \rrbracket := \{a, a+1, \dots, b\}$ and for $n \in \mathbb{N}$, $\llbracket n \rrbracket := \llbracket 1, n \rrbracket.$ With this, unless otherwise specified we take the state space $\mathcal{X} = \llbracket n \rrbracket$, and we write $\mathcal{P}(\mathcal{X})$ to be the set of all strictly positive probability masses on $\mathcal{X}$. That is, $\pi \in \mathcal{P}(\mathcal{X})$ satisfies $\pi(x) > 0$ for all $x \in \mathcal{X}.$

For any given $\pi \in \mathcal{P}(\mathcal{X}),$ we take $\Pi$ to be the $n \times n$ transition matrix on $\mathcal{X}$ with all rows equal to $\pi$. We further define $\mathcal{S}(\pi)$ as the set of all $\pi$-stationary transition matrices. Equivalently, for $P \in \mathcal{S}(\pi)$, we have that $\pi P = \pi$. 

We will also use $\mathcal{L}(\pi)$ to denote the set of all $\pi$-reversible transition matrices. That is, for $P \in \mathcal{L}(\pi)$, the detailed balance condition $\pi(x) P(x,y) = \pi(y) P(y,x)$ holds for any $x,y \in \mathcal{X}$.

We define $\ell^2(\pi)$ to be usual $\pi$-weighted Hilbert space, prescribed with the inner product $\langle \cdot, \cdot \rangle_\pi: \ell^2(\pi) \times \ell^2(\pi) \to \mathbb{R}$ given by 
$$\langle f,g \rangle_\pi := \sum_{x\in \mathcal{X}} f(x)g(x)\pi(x).$$
The $\ell^2(\pi)$-norm is then $\|f\|_\pi = \langle f,f \rangle_\pi^{1/2}$. We let 
$$\ell^2_0(\pi) := \{f \in \ell^2(\pi): \pi(f) = 0\}$$
denote the zero-mean subspace. Note that any transition matrix $P$ can also be viewed as an operator on $\ell^2(\pi)$, since
$$Pf(x) = \sum_{y \in \mathcal{X}} P(x,y) f(y)$$
is also a function in $\ell^2(\pi)$.

For any $\pi$-stationary $P \in \mathcal{S}(\pi)$, we define $P^*$ to be the $\ell^2(\pi)$-adjoint of $P$. In particular, $P^* = P$ if and only if $P \in \mathcal{L}(\pi).$ 

\subsection{Gibbs kernel}
Suppose we have a partition of the state space $\mathcal{X} = \bigsqcup_{i=1}^k \mathcal{O}_i$. We recall the Gibbs kernel $G \in \mathcal{L}(\pi)$ associated with the partition $(\mathcal{O}_i)_{i=1}^k$, as introduced in \cite{GPG}:
\begin{equation}
    G(x,y) = \begin{cases}
\frac{\pi(y)}{\pi(\mathcal{O}(x))},& \mathrm{for}\ y \in \mathcal{O}(x),\\
0, & \mathrm{otherwise},
\end{cases}
\end{equation}
where $\mathcal{O}(x) = \mathcal{O}_i$ with $x \in \mathcal{O}_i$, and $\pi(\mathcal{O}(x)) := \sum_{z \in \mathcal{O}(x)} \pi(z)$.

It is easy to verify that $G$ is $\pi$-reversible and idempotent, that is, $G^2 = G$. At any given state $x \in \mathcal{O}(x)$, one may view $G(x,\cdot)$ as a resampling from the target distribution restricted to the orbit $\mathcal{O}(x)$. 

\subsection{Projection chains}
Another important kernel which will be discussed extensively is the projection chain of a Markov kernel $P$. We recall from \cite{Jerrum_2004} that for any given partition $\mathcal{X} = \bigsqcup_{i=1}^k \mathcal{O}_i$ of the state space, we define $\overline{P}: \llbracket k \rrbracket \times \llbracket k \rrbracket \to [0,1]$ by
\begin{align} \label{eq:Pbar}
    \overline{P}(i,j) := \frac{1}{\pi(\mathcal{O}_i)} \sum_{\substack{x \in \mathcal{O}_i \\ y \in \mathcal{O}_j}} \pi(x) P(x,y).
\end{align}

Note that $\overline{P}$ has stationary distribution $\overline{\pi} = (\pi(\mathcal{O}_1), \dots, \pi(\mathcal{O}_k)).$ Further, if $P$ is $\pi$-reversible, then $\overline{P}$ will also be $\overline{\pi}$-reversible as well. 

Notably, the projection chain acts on the smaller state space $\{\mathcal{O}_i: i \in \llbracket k\rrbracket\}$. This dimensional reduction is a key motivation for linking group-averaged kernels with projection chains, as it allows for significantly simpler analysis on the quotient space. 

In the subsequent sections, we will relate additive mixtures of kernels to lifted constructions and investigate their behaviour through projection chains.

\subsection{Cheeger's constant and eigenvalues}
We briefly recall notions of eigenvalues and Cheeger's constant, which will be discussed extensively in relation to the Frobenius norm. 

Define the classical Cheeger's constant of a Markov chain $P$ to be 
\begin{align*}
    \Phi^*(P) := \min_{S;~0 < \pi(S) \leq \frac{1}{2}} \dfrac{\sum_{x \in S;\,y \in S'} \pi(x)P(x,y)}{\pi(S)}.
\end{align*}
A Cheeger's cut of $P$ is then any set $S^*$ attaining the minimum, that is, 
$$S^* \in \argmin_{S;~0 < \pi(S) \leq \frac{1}{2}} \frac{\sum_{x \in S;\,y \in S'} \pi(x)P(x,y)}{\pi(S)}.$$
Note that $S' = \mathcal{X} \setminus S$ refers to the complement of $S$. 

A similar notion called the symmetrised Cheeger's constant is defined in \cite{Montenegro} as 
\begin{align*}
    \phi^*(P) := \min_{S;~0 < \pi(S) < 1} \dfrac{\sum_{x \in S;\,y \in S'} \pi(x)P(x,y)}{\pi(S)\pi(S')},
\end{align*}
with 
$$U^* \in \argmin_{S;~0 < \pi(S) < 1} \frac{\sum_{x \in S;\,y \in S'} \pi(x)P(x,y)}{\pi(S)\pi(S')}$$
as a symmetrised Cheeger's cut of $P$.

Next, for any self-adjoint matrix $M \in \mathbb{R}^{n \times n}$, we use $(\lambda_i = \lambda_i(M))_{i=1}^n$ to denote the eigenvalues in non-increasing order counted with multiplicities. That is, 
\begin{align*}
    \lambda_1(M) \geq \lambda_2(M) \geq \ldots \geq \lambda_n(M).
\end{align*}
For $P \in \mathcal{L}(\pi)$, we write $\gamma(P) := 1 - \lambda_2(P)$ to be the right spectral gap of $P$.

Further, we define $\lambda^*(M) = \max\{|\lambda_2(M)|, |\lambda_n(M)|\}$ to be the second-largest eigenvalue in modulus (SLEM). For $P \in \mathcal{L}(\pi)$, the absolute spectral gap is then given by $\gamma^*(P) = 1 - \lambda^*(P).$

The classical Cheeger's inequality relates the Cheeger's constant with the spectral gap of a ergodic sampler $P \in \mathcal{L}(\pi)$ by 
\begin{align}\label{eq:Cheegerineq}
    \dfrac{\Phi^{*2}(P)}{2} \leq \gamma(P) \leq 2 \Phi^*(P).
\end{align}
A proof can be found in \cite{roch_mdp_2024}, Theorem 5.3.5. 

Another immediate inequality concerning the Cheeger's constant and its symmetric counterpart is given by 
\begin{align*}
    \phi^*(P) \leq 2 \Phi^*(P).
\end{align*}




\section{Construction of a lifted kernel $Q_\alpha$ and $A_\alpha$ as its projection} \label{sec: lift}
In this section, we discuss the construction and interpretation of the mixture $1/2(P+G)$, which will form the basis of our subsequent analysis. 

For any given distribution $\pi \in \mathcal{P}(\mathcal{X})$, consider some sampler $P \in \mathcal{S}(\pi)$ and the Gibbs kernel $G$ induced by the partition $\mathcal{X} = \bigsqcup_{i=1}^k \mathcal{O}_i$.

For $0 \leq \alpha \leq 1$, we define the additive kernel 
$$A_\alpha = A_\alpha(\mathcal{O}_1,\ldots,\mathcal{O}_k):= \alpha P + (1-\alpha)G,$$
and for the case $\alpha = 1/2$, we shall simply write $A = A_{1/2}$. It follows that $A$ is also $\pi$-stationary. In the special case of $k = n$ and $\mathcal{O}_i = \{i\}$ for $i \in \llbracket n \rrbracket$, we see that $G = I$ and hence $A_{1/2}$ is the lazified version of $P$.

We now shift our attention to the lifted space $\widetilde{\mathcal{X}} := \mathcal{X} \times \{-1, +1\}$. Define a Markov chain $Q_\alpha$ on $\widetilde{\mathcal{X}}$ as follows. From the current state $(x,i) \in \widetilde{\mathcal{X}}$:
\begin{enumerate}
    \item Sample the auxiliary variable:
    $$\sigma \sim R := \begin{cases}
    +1,& \mathrm{with\ probability}\ 1-\alpha, \\
    -1,& \mathrm{with\ probability}\ \alpha,
    \end{cases}$$
    where $0 \leq \alpha \leq 1$ is some fixed parameter.
    \item Update the $\mathcal{X}$-coordinate:
    \begin{align*}
        y \sim \begin{cases}
            G(x, \cdot),& \sigma = +1,\\
            P(x, \cdot),& \sigma = -1. 
        \end{cases}
    \end{align*}
    \item Set $j = \sigma$, that is, we move from $(x,i)$ to $(y,j)$.
\end{enumerate}

One can verify that the transition matrix $Q_\alpha$ is thus defined to be 
\begin{align}
    Q_\alpha((x,i), (y,j)) = \begin{cases}
        (1-\alpha) G(x,y),& j=+1,\\
        \alpha P(x,y),& j=-1.
    \end{cases}
\end{align}

\begin{prop} \label{prop:Q stationary}
The Markov chain $Q_\alpha$ admits $\widetilde{\pi}_\alpha := \pi \otimes R$ as its stationary distribution.
\end{prop}

\begin{proof}
For any $(y,j) \in \widetilde{\mathcal{X}}$, 
\begin{align*}
    \sum_{x,i} \widetilde{\pi}_\alpha(x,i) Q_\alpha((x,i),(y,j)) &= \sum_x \pi(x) \left(\sum_i R(i) \right) Q_\alpha((x,\cdot), (y,j))\\
    &= \sum_x \pi(x) Q_\alpha((x,\cdot), (y,j)),
\end{align*}
where in the first equality we make use of the fact that $Q_\alpha((x,i), (y,j))$ is independent of $i$. Now consider the two possibilities of $j$:
$$\sum_{x,i} \widetilde{\pi}_\alpha(x,i) Q_\alpha((x,i),(y,j)) = \begin{cases}
    (1-\alpha)\sum_x \pi(x) G(x,y),& j = +1,\\
    \alpha \sum_x \pi(x) P(x,y),& j = -1.
\end{cases}$$
Since both $P, G \in \mathcal{S}(\pi)$, we get 
$$\sum_{x,i} \widetilde{\pi}_\alpha(x,i) Q_\alpha((x,i),(y,j)) = \begin{cases}
    (1-\alpha)\pi(y),& j = +1,\\
    \alpha \pi(y),& j = -1,
\end{cases}$$
which is equivalent to $\widetilde{\pi}_\alpha(y,j).$
\end{proof}

Now consider the marginal chain of $Q_\alpha$ which acts only on $\mathcal{X}$, in which we denote by $Q^{(1)}_\alpha$ (see, e.g., \cite{choi2026geometryfactorizationmultivariatemarkov} for related lifted constructions and marginal chains). Given that $Q_\alpha((x,i), (y,j))$ is independent of $i$, 

\begin{align*}
    Q^{(1)}_\alpha(x,y) := \sum_{j \in \{-1, +1\}} Q_\alpha((x,i), (y,j)) = A_\alpha(x,y)
\end{align*}

In the case where $\alpha = 1/2$, the distribution $R$ reduces to the uniform distribution on $\{+1, -1\}$. We shall denote $Q = Q_{1/2}$, and it follows that $Q^{(1)} = A$. 

Hence, one can view the additive mixture $A_\alpha = \alpha P + (1-\alpha)G$ as the projection of the lifted kernel $Q_\alpha$ onto $\mathcal{X}$. In particular, the randomness in selecting between $P$ and $G$ at each step of $A_\alpha$ can be interpreted as arising from an auxiliary variable in the augmented space $\widetilde{\mathcal{X}}$.

\section{Frobenius norm} \label{sec: frobnorm}
We begin by introducing the Frobenius norm with respect to a distribution $\pi \in \mathcal{P}(\mathcal{X})$. For any real $n \times n$ matrices $M,N$, we define the Frobenius inner product with respect to $\pi$ as 
$$\langle M,N \rangle_{F,\pi} := \Tr(M^*N),$$
where $\Tr(M)$ is the trace of $M$. Section 2 of \cite{submodular} verifies that $\langle \cdot,\cdot \rangle_{F,\pi}$ is in fact, an inner product. 

The induced Frobenius norm is thus given by $\| M\|_{F, \pi}^2 := \langle M,M\rangle_{F,\pi}$.

The Frobenius norm may be understood as measuring discrepancy from stationarity in an averaged sense over all directions, in contrast to the spectral norm, which measures the discrepancy in the worst direction. Results concerning the spectral norm of $A_\alpha$ can be found in Section 2.5 of \cite{GPG}. In the present work, the Frobenius objective is particularly useful because, as shown below, the optimisation of Frobenius distance between $A_\alpha$ and $\Pi$ can be related to the symmetrised Cheeger's constant in the two-block case.

We first relate the compositions $GP$ and $PG$ to the projection chain $\overline{P}$ via trace identities.
\begin{lemma} \label{lem: Tr(GP)}
    Let $P \in \mathcal{S}(\pi)$ and $G$ be the Gibbs kernel induced by the partition $\mathcal{X} = \bigsqcup_{i=1}^k \mathcal{O}_i$. Then 
    $$\Tr(GP) = \Tr(PG) = \Tr(\overline{P}).$$
\end{lemma}

\begin{proof}
The first equality follows directly from the cyclic property of trace. For the second equality, 
    \begin{align*}
        \Tr(PG) &= \sum_{x\in \mathcal{X}} PG(x,x) \\
        &= \sum_{i=1}^k \frac{1}{\pi(\mathcal{O}_i)} \sum_{x,y \in \mathcal{O}_i} \pi(x)P(x,y) \\
        &= \sum_{i=1}^k \overline{P}(i,i)\\
        &= \Tr(\overline{P}).
    \end{align*}
\end{proof}

With the above lemma, we can relate the Frobenius distance between $A$ and $\Pi$ to the trace of $\overline{P}$.

\begin{prop} \label{prop: frob A_alpha}
     Let $P \in \mathcal{L}(\pi)$, $G$ be the Gibbs kernel induced by the partition $\mathcal{X} = \bigsqcup_{i=1}^k \mathcal{O}_i$ and recall that we define $A_\alpha = \alpha P + (1-\alpha)G.$ Then for $\alpha \in [0,1]$,
     $$\| A_\alpha - \Pi\|_{F, \pi}^2 = 2\alpha(1-\alpha)\Tr(\overline{P}) + \alpha^2 \Tr(P^2) + (1-\alpha)^2 k - 1.$$
\end{prop}

\begin{proof}
    Since both $P$ and $G$ are $\pi$-reversible, so will $A_\alpha$. Hence 
    \begin{align*}
        \| A_\alpha - \Pi\|_{F, \pi}^2 = \Tr((A_\alpha-\Pi)^2) = \Tr(A_\alpha^2) - 1. 
    \end{align*}
    Expanding out $A_\alpha^2$, we then obtain
    \begin{align*}
        \Tr(A_\alpha^2) &= \alpha^2\Tr(P^2) + 2\alpha(1-\alpha)\Tr(PG) + (1-\alpha)^2\Tr(G)\\
        &= \alpha^2 \Tr(P^2) + 2\alpha(1-\alpha)\Tr(\overline{P}) + (1-\alpha)^2k,
    \end{align*}
    where we use the results of Lemma \ref{lem: Tr(GP)}, and the fact that $\Tr(G) = k$.
\end{proof}

Specifically for the case $A = A_{1/2}$, we have the precise result:

\begin{corollary} \label{cor: frob and pbar}
    Under the same settings as Proposition \ref{prop: frob A_alpha},
    $$\| A - \Pi\|_{F, \pi}^2 = \frac{1}{4} \Tr(P^2) + \frac{1}{2}\Tr(\overline{P}) + \frac{k}{4} - 1.$$
\end{corollary} 

\subsection{Bound and decay rate of Frobenius objective}
This subsection aims to derive an explicit decay rate of the Frobenius objective in terms of the absolute spectral gap of $P.$ 

\begin{lemma} \label{lem: SLEM A_alpha}
    Let $\alpha \in (0,1)$, $P \in \mathcal{L}(\pi)$ and $G$ be a Gibbs kernel. Then 
    $$\lambda^*(A_\alpha) = \lambda^*(\alpha P + (1-\alpha) G) \leq 1 - \alpha\gamma^*(P).$$
\end{lemma}

\begin{proof}
    Recall that the spectrum of $G \subseteq \{0,1\}.$ 
    First, suppose $\lambda^*(A_\alpha) = \lambda_2(A_\alpha) \geq 0.$ By Weyl's inequality, 
    \begin{align*}
        \lambda_2(A_\alpha) &= \lambda_2(\alpha P + (1-\alpha)G)\\
        &\leq \alpha\lambda_2(P) + (1-\alpha)\lambda_1(G)\\
        &\leq \alpha\lambda^*(P) + 1 - \alpha\\
        &= 1 - \alpha\gamma^*(P)
    \end{align*}
    Else, suppose $\lambda^*(A_\alpha) = |\lambda_n(A_\alpha)| = - \lambda_n(A_\alpha).$ Weyl's inequality gives 
    $$\lambda_n(A_\alpha) \geq \alpha\lambda_n(P) + (1-\alpha)\lambda_n(G) \geq \alpha\lambda_n(P).$$
    Then 
    $$- \lambda_n(A_\alpha) \leq -\alpha\lambda_n(P) \leq \alpha\lambda^*(P) \leq 1 - \alpha\gamma^*(P).$$
    In either case, the inequality holds. 
\end{proof}

\begin{prop} \label{prop: decay frob}
    Take any $A_\alpha = \alpha P + (1-\alpha)G$ with $G$ as the Gibbs kernel, and $P$ a $\pi$-reversible kernel. Then for any positive integer $l \geq 2$ and $\alpha \in (0,1)$,
    \begin{equation}
        \| A_\alpha^l - \Pi\|_{F,\pi}^2 \leq (1-\alpha \gamma^*(P))^{2(l-1)} \| A_\alpha - \Pi\|_{F,\pi}^2
    \end{equation}
\end{prop}

\begin{proof} 
    It follows from $\pi$-stationarity of $A_\alpha$, and the definition of the Frobenius norm that
    \begin{align*}
        \| A_\alpha^l - \Pi\|_{F,\pi}^2 &= \Tr((A_\alpha^l - \Pi)^2) = \Tr(A_\alpha^{2l}) - 1.
    \end{align*}

    Since 
    $$\Tr(A_\alpha^{2l}) = \sum_{i=1}^n \lambda_i(A_\alpha^{2l})$$
    and 
    $$\lambda_1(A_\alpha^{2l}) = 1,$$
    the objective can be written as 
    $$\| A_\alpha^l - \Pi\|_{F,\pi}^2 = \sum_{i=2}^n \lambda_i(A_\alpha^{2l}) = \sum_{i=2}^n \lambda_i(A_\alpha)^{2l}.$$
    One then bounds the eigenvalues by the SLEM, giving us 
    \begin{align*}
        \| A_\alpha^l - \Pi\|_{F,\pi}^2 \leq \lambda^*(A_\alpha)^{2(l-1)} \sum_{i=2}^n \lambda_i(A_\alpha)^2.
    \end{align*}
    Finally, notice that $\sum_{i=2}^n \lambda_i(A_\alpha)^2 = \| A_\alpha - \Pi\|_{F,\pi}^2$ and using the upper bound of $\lambda^*(A_\alpha)$ shown in Lemma \ref{lem: SLEM A_alpha}, we complete the proof.
    
\end{proof}

The results of Proposition \ref{prop: decay frob} thus show that with every increasing step of the kernel $A_\alpha$, the Frobenius distance decreases with a geometric rate of $(1-\alpha\gamma^*(P))^2 < 1$ when $P$ is assumed to be ergodic. We can further obtain a crude upper bound for $\| A_\alpha^l - \Pi\|_{F,\pi}^2$.

\begin{corollary}
    Under the same setting as Proposition \ref{prop: decay frob}, 
    $$\| A_\alpha^l - \Pi\|_{F,\pi}^2 \leq (n-1)(1-\alpha\gamma^*(P))^{2l},$$
    where $n = |\mathcal{X}|.$
\end{corollary}

\begin{proof}
    The result follows directly from Lemma \ref{lem: SLEM A_alpha} and Proposition \ref{prop: decay frob}, noting that 
    $$\sum_{i=2}^n \lambda_i(A_\alpha)^2 \leq \sum_{i=2}^n \lambda^*(A_\alpha)^2$$
\end{proof}

We note that Proposition \ref{prop: decay frob} provides a conservative geometric decay bound for any choice of partition, rather than a uniform improvement over $P$. Indeed, since $0<\alpha<1$, the bound $(1-\alpha\gamma^*(P))^2$ is weaker than the corresponding bound
$(1-\gamma^*(P))^2$ for $P$. On the other hand, comparing Proposition \ref{prop: frob A_alpha} with $\|P-\Pi\|_{F,\pi}^2=\Tr(P^2)-1$ shows that the Frobenius distance of $A_\alpha$ to $\Pi$ depends on the choice of partition. This motivates the subsequent optimisation over partitions.

\subsection{Frobenius optimisation}

It is clear from Proposition \ref{prop: frob A_alpha} that for a fixed sampler $P \in \mathcal{L}(\pi)$, a fixed choice of $\alpha \in (0,1)$, and a fixed number of orbits $k$, the choice of partition will only affect the Frobenius distance via the term $\Tr(\overline{P})$. Hence, the minimisation problem reduces to 
$$\argmin_{\mathcal{O}_1,\ldots,\mathcal{O}_k} \| A_\alpha - \Pi\|_{F, \pi}^2 = \argmin_{\mathcal{O}_1,\ldots,\mathcal{O}_k} \Tr(\overline{P}),$$
where the minimisation is over the choice of partition of $\mathcal{X}$.

For our subsequent analysis, we shall consider only the partition $\mathcal{X} = S \sqcup S',$  given that $S \neq \emptyset$ or $\mathcal{X}$. 

Our motivation for studying two-block partitions is primarily analytical. Such partitions arise naturally when the state space is separated by a binary or thresholded summary statistic, for example positive versus negative magnetisation, low versus high energy, or an event and its complement. Moreover, as we will see shortly, the two-block case leads to a direct connection with the symmetrised Cheeger's constant and admits a submodular optimisation structure. We note, however, that two-block Gibbs updates may be computationally expensive to implement. Nonetheless, the two-block case provides an analytically tractable setting for understanding structural features of partitions that are favourable for the Frobenius objective.

From here on, whenever we study two-block partitions $\mathcal X=S\sqcup S'$, we write $A_\alpha(S):=A_\alpha(S,S')$ and $A(S):=A(S,S')$ to stress the dependence on $S$, with $S'=\mathcal X\setminus S$ suppressed from the notation. The notation $A_\alpha$ will continue to refer to the additive kernel associated with an arbitrary $k$-block partition.

\begin{prop} \label{prop: Tr(Pbar)}
    Let $P \in \mathcal{L}(\pi),$ and consider the projection chain $\overline{P}$ induced by the partition $\mathcal{X} = S \sqcup S', S \neq \emptyset, \mathcal{X}$. Then 
    $$\Tr(\overline{P}) = 2 - \frac{1}{\pi(S)\pi(S')}\sum_{\substack{x \in S\\ y \in S'}} \pi(x) P(x,y).$$
\end{prop}

\begin{proof}
    \begin{align*}
        \Tr(\overline{P}) &= \frac{1}{\pi(S)}\sum_{x,y \in S}\pi(x)P(x,y) + \frac{1}{\pi(S')}\sum_{x,y \in S'}\pi(x)P(x,y)\\
        &= \frac{1}{\pi(S)} \bigg(\pi(S) -  \sum_{\substack{x \in S \\ y \in S'}} \pi(x) P(x,y)\bigg) + \frac{1}{\pi(S')} \bigg(\pi(S') -  \sum_{\substack{x \in S' \\ y \in S}} \pi(x) P(x,y)\bigg)\\
        &= 2 - \frac{1}{\pi(S) \pi(S')} \sum_{\substack{x \in S \\ y \in S'}} \pi(x) P(x,y).
    \end{align*}
    Note that the second to last equality requires $P$ to be $\pi$-reversible.
\end{proof}

The results of Propositions \ref{prop: frob A_alpha} and \ref{prop: Tr(Pbar)} give us the following corollaries.

\begin{corollary} \label{cor: frob and g(S)}
     For any $\pi$-reversible $P$, $S \neq \mathcal{X}, \emptyset$ and $\mathcal{X} = S \sqcup S'$, we have that for any fixed $\alpha \in (0,1),$ the non-trivial set $S \subset \mathcal{X}$ that minimises $\| A_\alpha(S) - \Pi\|_{F, \pi}^2$ maximises the function 
     \begin{equation} \label{eq: g(S)}
         g(S) = g(S,P) := \frac{1}{\pi(S) \pi(S')} \sum_{\substack{x \in S \\ y \in S'}} \pi(x) P(x,y).
     \end{equation}

     The other direction holds as well. Concretely, we have that 
     \begin{align*}
         \argmin_{S \neq \mathcal{X}, \emptyset} \|A_\alpha(S) - \Pi\|_{F, \pi}^2 &= \argmax_{S \neq \mathcal{X}, \emptyset} g(S),\\
         \argmax_{S \neq \mathcal{X}, \emptyset} \|A_\alpha(S) - \Pi\|_{F, \pi}^2 &= \argmin_{S \neq \mathcal{X}, \emptyset} g(S).
     \end{align*}
\end{corollary}

By noting that the set function $g$ given in \eqref{eq: g(S)} is symmetric about $S$ (i.e. $g(S) = g(S')$ for $\pi$-reversible $P$), the minimisation/maximisation of $g$ can be reduced to the space 
$$\mathcal{A} = \{S \subset \mathcal{X}: 0 < \pi(S) \leq 1/2\}.$$

\begin{corollary}
    Under the same settings as Corollary \ref{cor: frob and g(S)}, we have that the symmetrised Cheeger's minimiser corresponds to the partition which maximises the Frobenius distance $\|A_\alpha(S) - \Pi\|_{F, \pi}^2$. Equivalently, 
    $$\max_{S \neq \mathcal{X}, \emptyset} \|A_\alpha(S) - \Pi\|_{F, \pi}^2 = \alpha^2\Tr(P^2) - 2\alpha(1-\alpha)\phi^*(P) + 1 - 2\alpha^2.$$
\end{corollary}

The above results provide an intuitive guideline for the choice of partition. Since the symmetrised Cheeger's cut identifies a bottleneck of $P$, our results show that a favourable partition should be one that bridge across the bottlenecks of the base sampler $P$. In contrast, choosing the partition along the bottleneck itself gives the worst possible two-block partition for the Frobenius objective.




\subsection{Singleton approximation of Frobenius distance to stationarity}
While $g(S)$ can be used as to exactly optimise the Frobenius objective, the need for combinatorial optimisation implies that the problem is often intractable when $|\mathcal{X}|$ is large. Instead, we shall approximate $g(S)$ with the functional 
\begin{equation}\label{eq:h(S)}
    h(S) = h(S,P) := \frac{1}{\pi(S)} \sum_{\substack{x \in S\\ y \in S'}} \pi(x)P(x,y).
\end{equation}

By Proposition 4.8 and Lemma 4.9 of \cite{submodular}, we have that for any ergodic, $\pi$-stationary $P$, define
$$x^* = x^*(P) \in \argmax_{0 < \pi(x) < 1} \left(1-P(x,x)\right).$$

Then for $|\mathcal{X}| \geq 2$,
$$U^* = \{x^*\} \in \argmax_{S \neq \mathcal{X}, \emptyset} h(S)\quad \mathrm{and}\quad U^*\in \mathcal{A}.$$

Further, for any $S \in \mathcal{A}$, we have 
\begin{equation*}
    h(S) \leq g(S) \leq 2h(S).
\end{equation*}

With these, we now show the following approximation result.
\begin{prop}
    Consider some ergodic $P \in \mathcal{L}(\pi)$ and let $|\mathcal{X}| \geq 2$. Recall the function h introduced in \eqref{eq:h(S)}. Any solution
    $$U^* = \{x^*\} \in \argmax_{S\in \mathcal{A}} h(S)$$
    is an additive $2\alpha(1-\alpha)$-approximate minimiser of the squared Frobenius distance $\|A_\alpha(S) - \Pi\|_{F,\pi}^2$. That is, 
    $$S^* \in \argmin_{S\in \mathcal{A}} \|A_\alpha(S) - \Pi\|_{F,\pi}^2$$ 
    satisfies 
    $$ 0\leq \|A_\alpha(U^*) - \Pi\|_{F,\pi}^2 - \|A_\alpha(S^*) - \Pi\|_{F,\pi}^2 \leq 2\alpha(1-\alpha).$$
\end{prop}

\begin{proof}
    The first inequality follows from the minimality of $S^*$. For the second inequality, first note that 
    by maximality of $U^*$, $h(U^*) \geq h(S^*).$ Then
    \begin{align*}
        g(S^*) - g(U^*) &\leq 2h(S^*) - h(U^*) \leq h(U^*).
    \end{align*}
    The result follows since
    \begin{align*}
        \|A_\alpha(U^*) - \Pi\|_{F,\pi}^2 - \|A_\alpha(S^*) - \Pi\|_{F,\pi}^2 &= 2\alpha(1-\alpha)(g(S^*) - g(U^*))\\
        &\leq 2\alpha(1-\alpha)h(U^*)
    \end{align*}
    and $h(U^*) = 1-P(x^*, x^*) \leq 1.$
\end{proof}
We remark that the quantity $2\alpha(1-\alpha)$ is in fact, bounded above by the constant $1/2$.

Using the approximation $h(S)$ thus reduces the optimisation over all non-trivial subsets of $\mathcal{X}$ to a search over singletons. This effectively reduces the search from one that is exponential in $|\mathcal{X}|$ to one that is instead linear in $|\mathcal{X}|$.

Lastly, we present as a corollary for the special case where we restrict $S$ to be a singleton. 

\begin{corollary}
    For $P \in \mathcal{L}(\pi)$, under the additional constraint $|S| = 1$, 
    $$\argmin_{\substack{S \neq \mathcal{X}\\ |S| = 1}} \|A_\alpha(S) - \Pi\|^2_{F, \pi} = \argmax_{\substack{S \neq \mathcal{X}\\ |S| = 1}} g(S) = \argmax_{x \in \mathcal{X}} \frac{1-P(x,x)}{1-\pi(x)}.$$
\end{corollary}

\subsection{Submodular optimisation of Frobenius objective}

We now turn to submodular optimisation in our attempt to minimise the objective $\|A_\alpha(S) - \Pi\|_{F,\pi}^2.$ Throughout this subsection, let $\mathcal D:=\{S\subseteq\mathcal X:0<\pi(S)<1\}.$ When we say a set function $f$ is submodular in this subsection, we mean that $f: \mathcal{D} \to \mathbb{R}$ and that the submodularity inequality 
$$f(A)+f(B)\geq f(A\cup B)+f(A\cap B)$$
holds for all pairs $A,B \in \mathcal{D}$ such that $0 < \pi(A\cup B), \pi(A\cap B) < 1.$ Similarly, supermodularity implies the same with the inequality reversed. We refer the reader to \cite{Krause_Golovin_2014} for a more comprehensive overview of submodular functions and maximisation.

The results in Propositions \ref{prop: frob A_alpha} and \ref{prop: Tr(Pbar)} show that one can write 
$$\|A_\alpha(S) - \Pi\|_{F,\pi}^2 = \alpha^2\Tr(P^2) - 2\alpha(1-\alpha)g(S) + 1 - 2\alpha^2.$$
It can then be shown that $g(S)$ can be written as a difference of two submodular functions, and by extension, so can $\|A_\alpha(S) - \Pi\|_{F,\pi}^2$.

\begin{prop}
    For any subset $S \in \mathcal D$, and $P \in \mathcal{L}(\pi)$, 
    \begin{align*}
        g(S) = &\frac{1}{\pi(S) \pi(S')} \sum_{\substack{x \in S \\ y \in S'}} \pi(x) P(x,y)\\
    = &\underbrace{\dfrac{1}{\pi(S)\pi(S')} - 2}_{\text{supermodular}} - \underbrace{\dfrac{1}{\pi(S)} \sum_{x,y \in S'} \pi(x) P(x,y)}_{\text{supermodular}} \\
            &\quad - \underbrace{\dfrac{1}{\pi(S')} \sum_{x,y \in S} \pi(x) P(x,y).}_{\text{supermodular}}
    \end{align*}
\end{prop}

\begin{proof}
    Recall the definition of $g(S)$ given in \eqref{eq: g(S)}. Then 
    \begin{align*}
        g(S) &= \frac{1}{\pi(S) \pi(S')} \sum_{\substack{x \in S \\ y \in S'}} \pi(x) P(x,y)\\
        &= \frac{1}{\pi(S)}  \sum_{\substack{x \in S \\ y \in S'}} \pi(x) P(x,y) + \frac{1}{\pi(S')} \sum_{\substack{x \in S \\ y \in S'}} \pi(x) P(x,y).\\
    \end{align*}
    By $\pi$-reversibility of $P$, the first term yields
    \begin{align*}
        \frac{1}{\pi(S)}  \sum_{\substack{x \in S \\ y \in S'}} \pi(x) P(x,y) &= \frac{1}{\pi(S)}  \sum_{\substack{x \in S' \\ y \in S}} \pi(x) P(x,y)\\
        &= \frac{1}{\pi(S)} \sum_{x\in S'} \pi(x) \left[1 - \sum_{y \in S'}P(x,y)\right]\\
        &= \frac{\pi(S')}{\pi(S)} - \frac{1}{\pi(S)}\sum_{\substack{x, y \in S'}} \pi(x) P(x,y).
    \end{align*}
    The second term can be similarly written as 
    \begin{align*}
        \frac{1}{\pi(S')} \sum_{\substack{x \in S \\ y \in S'}} \pi(x) P(x,y) = \frac{\pi(S)}{\pi(S')} - \frac{1}{\pi(S')}\sum_{\substack{x, y \in S}} \pi(x) P(x,y).
    \end{align*}
    The equality holds by noting that 
    $$\frac{\pi(S')}{\pi(S)} + \frac{\pi(S)}{\pi(S')} = \frac{1-2\pi(S)\pi(S')}{\pi(S)\pi(S')} = \frac{1}{\pi(S)\pi(S')} - 2.$$

    Lemma 6.9 of \cite{submodular} shows that the functions 
    \begin{align*}
        \{S \subseteq \mathcal{X}; 0 < \pi(S) < 1\} \ni S &\mapsto \dfrac{1}{\pi(S)} \sum_{x \in S'} \sum_{y \in S'} \pi(x) P(x,y) \\
        \{S \subseteq \mathcal{X}; 0 < \pi(S) < 1\} \ni S &\mapsto \dfrac{1}{\pi(S')} \sum_{x \in S} \sum_{y \in S} \pi(x) P(x,y)
    \end{align*}
    are supermodular in $S$. The same result also shows that 
    \begin{align*}
        \{S \subseteq \mathcal{X}; 0 < \pi(S) < 1\} \ni S &\mapsto 3 - \dfrac{1}{\pi(S)\pi(S')}
    \end{align*}
    is submodular. Since addition of constant does not change submodularity, it follows that 
    $$\{S \subseteq \mathcal{X}; 0 < \pi(S) < 1\} \ni S \mapsto \dfrac{1}{\pi(S)\pi(S')} -2$$
    is supermodular. 
\end{proof}

As a corollary, we present the full difference-of-submodular decomposition of $\|A_\alpha(S) - \Pi\|_{F,\pi}^2.$

\begin{corollary} \label{cor: frob submodular}
    For any subset $S \in \mathcal{D}$ and $P \in \mathcal{L}(\pi)$, 
    \begin{align*}
        \|A_\alpha(S) - \Pi\|_{F,\pi}^2 &= \alpha^2\Tr(P^2) - 2\alpha(1-\alpha)g(S) + 1 - 2\alpha^2\\
         = &\underbrace{2\alpha(1-\alpha) \left(\dfrac{1}{\pi(S)} \sum_{x,y \in S'} \pi(x) P(x,y) + \dfrac{1}{\pi(S')} \sum_{x,y \in S} \pi(x) P(x,y)\right)}_{\text{supermodular}}\\
         &-\underbrace{\left(\dfrac{2\alpha(1-\alpha)}{\pi(S)\pi(S')} + 6\alpha^2 - 4\alpha - 1 - \alpha^2\Tr(P^2)\right).}_{\text{supermodular}}
    \end{align*}
\end{corollary}

A majorisation-minimisation (MM) algorithm can then be developed based on the results of Corollary \ref{cor: frob submodular}. One such possible majorisation uses the fact that since $(0,1) \ni t \mapsto 1/(t(1-t))$ is convex and differentiable, for $S,S^0 \subseteq \mathcal{X}$ with $S, S^0 \neq \emptyset,\mathcal{X}$, the supporting hyperplane theorem gives
\begin{align*}
    \dfrac{1}{\pi(S)\pi(S')} \geq \dfrac{1}{\pi(S^0)\pi(S^{0'})} + \dfrac{2\pi(S^0) - 1}{\pi(S^0)^2 (1 - \pi(S^0))^2}(\pi(S) - \pi(S^0)).
\end{align*}
Note that the term on the right-hand side is modular in $S$. One thus have 
\begin{align*}
        \|A_\alpha(S) - \Pi\|_{F,\pi}^2 &= \alpha^2\Tr(P^2) - 2\alpha(1-\alpha)g(S) + 1 - 2\alpha^2\\
         &\leq 2\alpha(1-\alpha) \left(\dfrac{1}{\pi(S)} \sum_{x,y \in S'} \pi(x) P(x,y) + \dfrac{1}{\pi(S')} \sum_{x,y \in S} \pi(x) P(x,y)\right)\\
         &-(2\alpha(1-\alpha)\left(\dfrac{1}{\pi(S^0)\pi(S^{0'})} + \dfrac{2\pi(S^0) - 1}{\pi(S^0)^2 (1 - \pi(S^0))^2}(\pi(S) - \pi(S^0))\right) \\
         &- 6\alpha^2 + 4\alpha + 1 + \alpha^2\Tr(P^2)\\
         &=:  \zeta(S;S^0) = \zeta(S,P^2,\pi;S^0),
\end{align*}
which is supermodular in $S$. 

Similar to the MM-type approach discussed in Section 6.4 of \cite{submodular},
one may use algorithms for supermodular minimisation, such as those proposed
in \cite{uriel_2011}, to generate candidate updates for the surrogate problem.

If the majorisation subproblem is solved exactly, namely if
$$S^1\in \argmin_{S\in\mathcal A}\zeta(S;S^0),$$
then since $S^0\in\mathcal A$, we have
$$\|A_\alpha(S^1)-\Pi\|_{F,\pi}^2 \leq \zeta(S^1;S^0) \leq \zeta(S^0;S^0) = \|A_\alpha(S^0)-\Pi\|_{F,\pi}^2.$$

Thus the ideal MM procedure produces a non-increasing sequence of Frobenius
objectives. Iterating this exact update, for any positive integer $l$,
$$S^l\in \argmin_{S\in\mathcal A}\zeta(S;S^{l-1}),$$
gives
$$\|A_\alpha(S^0)-\Pi\|_{F,\pi}^2 \geq \|A_\alpha(S^1)-\Pi\|_{F,\pi}^2 \geq \cdots \geq \|A_\alpha(S^l)-\Pi\|_{F,\pi}^2.$$

In practice, however, one may only obtain an approximate minimiser. Suppose
that a candidate $\widehat S^1\in\mathcal A$ satisfies
$$\zeta(\widehat S^1;S^0) \leq \min_{S\in\mathcal A}\zeta(S;S^0)+\varepsilon_0.$$

Then, since $S^0\in\mathcal A$, $$\|A_\alpha(\widehat S^1)-\Pi\|_{F,\pi}^2 \leq \zeta(\widehat S^1;S^0) \leq \|A_\alpha(S^0)-\Pi\|_{F,\pi}^2+\varepsilon_0.$$

To preserve exact monotonicity, one may use the acceptance rule
$$S^1 = \begin{cases}
    \widehat S^1, & \zeta(\widehat S^1;S^0)\leq \zeta(S^0;S^0),\\
    S^0, & \text{otherwise}.
    \end{cases}$$
    
Under this accepted update,
$$\|A_\alpha(S^1)-\Pi\|_{F,\pi}^2 \leq \|A_\alpha(S^0)-\Pi\|_{F,\pi}^2.$$

More generally, if $\widehat S^l$ is an approximate candidate generated from
$\zeta(\cdot;S^{l-1})$, and $S^l$ is defined by the corresponding acceptance
rule, then
$$\|A_\alpha(S^0)-\Pi\|_{F,\pi}^2 \geq \|A_\alpha(S^1)-\Pi\|_{F,\pi}^2 \geq \cdots \geq \|A_\alpha(S^l)-\Pi\|_{F,\pi}^2.$$

We remark that there are many other possible choices of majorisation/minorisation functions, and we refer readers to the related works by \cite{iyer2013algorithmsapproximateminimizationdifference}, where they discuss modular upper and lower bounds, with several other approximation procedures motivated by the MM algorithm. 

\section{KL divergence} \label{sec: KL}
We begin the section by recalling the definitions related to Kullback-Leibler (KL) divergence, followed by an analysis of the kernel $A_\alpha$ in terms of KL divergence to stationarity. 

For any two probability distributions on $\mathcal{X}$, the KL divergence of $\mu$ from $\nu$ is given by 
$$\KL{}(\mu \| \nu) := \sum_{x \in \mathcal{X}} \mu(x) \log \frac{\mu(x)}{\nu(x)}.$$

For Markov chains, suppose $P, Q \in \mathcal{S}(\pi)$, then the KL divergence of $P$ from $Q$ weighted by $\pi$ is defined to be 
$$\KL{\pi}(P\|Q) := \sum_{x,y \in \mathcal{X}} \pi(x) P(x,y) \log\frac{P(x,y)}{Q(x,y)}.$$

Note that we take $0 \log(0/a) := 0$ by convention for all $a \in [0,1]$ for the above definitions. 

We also define the Shannon entropy of any given distribution $\pi$ on $\mathcal{X}$ to be 
$$H(\pi) := -\sum_{x\in\mathcal X} \pi(x)\log\pi(x).$$

Our first result concerns the KL divergence of $Q_\alpha$ to its stationary chain $\widetilde{\Pi}_\alpha$, where $\widetilde{\Pi}_\alpha$ has all rows equal to $\widetilde{\pi}_\alpha$. 

\begin{prop} \label{prop: KL Q_alpha}
    For $P \in \mathcal{S}(\pi)$ and $G$ a Gibbs kernel defined on some partition $(\mathcal{O}_i)_{i=1}^k$, 
    \begin{equation}\label{eq: KL Q_alpha}
        \KL{\widetilde{\pi}_\alpha}(Q_\alpha\ \|\ \widetilde{\Pi}_\alpha) = (1-\alpha)\KL{\pi} (G\| \Pi) + \alpha \KL{\pi}(P \| \Pi). 
    \end{equation}
\end{prop}

\begin{proof}
    The result is immediate for the cases $\alpha = 0$ or $1$. For $\alpha \in (0,1)$, we have by definition that
    $$\KL{\widetilde{\pi}_\alpha}(Q_\alpha \| \widetilde{\Pi}_\alpha) = \sum_{(x,i)\in \widetilde{\mathcal X}} \widetilde{\pi}_\alpha(x,i) \KL{}(Q_\alpha((x,i),\cdot)\|\widetilde{\Pi}_\alpha((x,i),\cdot)).$$
    
    Since
    $$Q_\alpha((x,i),(y,j))=R(j)K_j(x,y) \quad \mathrm{and} \quad \widetilde{\Pi}_\alpha((x,i),(y,j))=R(j)\Pi(y),$$
    
    for each $(x,i)$, we have that
    \begin{align*}
    \KL{}(Q_\alpha((x,i),\cdot)\|\widetilde{\Pi}_\alpha((x,i),\cdot)) &= \sum_{j\in\{-1,+1\}}\sum_{y\in\mathcal X} R(j)K_j(x,y)\log\frac{R(j)K_j(x,y)}{R(j)\Pi(y)} \\
    &= \sum_{j\in\{-1,+1\}} R(j)\KL{}(K_j(x,\cdot)\|\pi).
    \end{align*}
    
    Hence,
    \begin{align*}
    \KL{\widetilde{\pi}_\alpha}(Q_\alpha \| \widetilde{\Pi}_\alpha)
    &= \sum_{x\in\mathcal X}\sum_{i\in\{-1,+1\}} \pi(x)R(i) \sum_{j\in\{-1,+1\}} R(j)\KL{}(K_j(x,\cdot)\|\pi) \\
    &= \sum_{x\in\mathcal X}\pi(x)\sum_{j\in\{-1,+1\}} R(j)\KL{}(K_j(x,\cdot)\|\pi)\\
    &= \sum_{j\in\{-1,+1\}} R(j) \sum_{x\in\mathcal X}\pi(x)\KL{}(K_j(x,\cdot)\|\pi).
    \end{align*}
    
    Recalling $K_{+1}=G$, $K_{-1}=P$, we conclude
    $$\KL{\widetilde{\pi}_\alpha}(Q_\alpha \| \widetilde{\Pi}_\alpha) = (1-\alpha)\KL{\pi}(G\|\Pi)+\alpha \KL{\pi}(P\|\Pi).$$
\end{proof}

Equation \eqref{eq: KL Q_alpha} of Proposition \ref{prop: KL Q_alpha} further implies that for a fixed sampler $P$, minimising the KL divergence of $Q_\alpha$ to stationarity is equivalent to minimising that of $G$ from its respective stationary distribution.  

\begin{lemma} \label{lem: KL G = H}
    Let $G$ be the Gibbs kernel defined by the partition $(\mathcal{O}_i)_{i=1}^k$. Then 
    $$\KL{\pi}(G\| \Pi) = H(\overline{\pi}),$$
    where we recall $\overline{\pi} = (\pi(\mathcal{O}_1),\ldots,\pi(\mathcal{O}_k)).$
\end{lemma}

\begin{proof}
    \begin{align*}
        \KL{\pi}(G\| \Pi) &= \sum_{x,y \in \mathcal{X}} \pi(x) G(x,y) \log \frac{G(x,y)}{\pi(y)}\\
        &= \sum_{i=1}^k \sum_{x,y \in \mathcal{O}_i} \frac{\pi(x)\pi(y)}{\pi(\mathcal{O}_i)} \log \frac{1}{\pi(\mathcal{O}_i)}\\
        &= -\sum_{i=1}^k \pi(\mathcal{O}_i) \log \pi(\mathcal{O}_i)\\
        &= H(\overline{\pi}). 
    \end{align*}
\end{proof}

\begin{prop} \label{prop: best O}
    Let $P \in \mathcal{S}(\pi)$ be some fixed sampler and $G$ the Gibbs sampler induced by some partition $(\mathcal{O}_i)_{i=1}^k$, where $n \geq k \geq 2.$. Suppose further that $\pi$ is ordered non-decreasingly, i.e. $\pi(1) \leq \pi(2) \leq \dots \leq \pi(n).$ Then for $\alpha \in [0,1)$, the partition
    \begin{equation}\label{eq: best O}
        (\mathcal{O}_i)_{i=1}^k = \begin{cases}
        \{ i \}, & \mathrm{if}\ 1 \leq i \leq k-1,\\
        \{ k, k+1, \dots, n\}, & \mathrm{if}\ i = k.
        \end{cases}
    \end{equation}
    belongs to the set of minimisers

    \begin{align*}
        \argmin_{\mathcal{O}_1 \neq \emptyset, \mathcal{X},\ldots,\mathcal{O}_k \neq \emptyset, \mathcal{X}} \KL{\widetilde{\pi}_\alpha}(Q_\alpha\ \|\ \widetilde{\Pi}_\alpha) 
        &= \argmin_{\mathcal{O}_1 \neq \emptyset, \mathcal{X},\ldots,\mathcal{O}_k \neq \emptyset, \mathcal{X}} \KL{\pi} (G\| \Pi) \\
        &= \argmin_{\mathcal{O}_1 \neq \emptyset, \mathcal{X},\ldots,\mathcal{O}_k \neq \emptyset, \mathcal{X}} H(\overline{\pi}).
    \end{align*}

    Moreover, the minimiser is unique up to permutation of block labels and ties in the probabilities $\pi(x)$. In particular, any minimising partition may be obtained by choosing $k-1$ states of smallest $\pi$-mass as singleton blocks, with the remaining states grouped into one block.
\end{prop}

\begin{proof}
    From Proposition \ref{prop: KL Q_alpha}, the first equality follows since the optimisation over the choice of orbits is independent of the KL divergence of $P$ from $\Pi$. The second optimisation uses the results of Lemma \ref{lem: KL G = H}. 
    
    Finally, the fact that $(\mathcal{O}_i)_{i=1}^k$ is a minimiser of $\KL{\pi}(G\|\Pi)$ follows from Proposition 7.1 of \cite{GPG}. We briefly recall the proof for completeness.

    If $k = n$, then $(\mathcal{O}_i)_{i=1}^k$ is the only partition of $\mathcal{X}$.
    
    Suppose $n > k$ and define $(\mathcal{C}_i)_{i=1}^k$ as an arbitrary partition of $\mathcal{X}$. Let $g(t) = t\log t$, and for any two blocks $\mathcal{C}_i, \mathcal{C}_j$ with total mass $M = \pi(\mathcal{C}_i)+\pi(\mathcal{C}_j)$, define $h(t)=g(t)+g(M-t)$. Note that $h$ is strictly convex on $(0,M)$ and achieves its maximum at the endpoints.

    Suppose there are two distinct blocks $\mathcal{C}_i$ and $\mathcal{C}_j$, both with more than one element. Let $x_m$ be the element within $\mathcal{C}_i \cup \mathcal{C}_j$ with the smallest probability.
    
    By strict convexity of $h$ on $(0,M)$,
    $$g(\pi(\mathcal{C}_i)) + g(\pi(\mathcal{C}_j)) = h(\pi(\mathcal{C}_i)) < h(\pi(x_m)) = g(\pi(x_m))+g(\pi((\mathcal{C}_i \cup \mathcal{C}_j) \setminus \{x_m\})).$$

    Thus replacing the pair $(\mathcal{C}_i,\mathcal{C}_j)$ by a new pair consisting of the singleton $\{x_m\}$ and the merged remainder $(\mathcal{C}_i \cup \mathcal{C}_j) \setminus \{x_m\}$ strictly decreases $H(\overline\pi)$. Iterating the argument implies that the minimum achieving partition must have only one non-singleton block.

    Suppose the partition now has only one non-singleton set $\mathcal{C}_i$, but among the singletons, there exist $\mathcal{C}_j = \{y\}$ such that $\pi(y) > \pi(x_m) = \min_{x \in \mathcal{C}_i} \pi(x).$
    
    The same convexity argument will then show that replacing $\mathcal{C}_i$ and $\{y\}$ by the pair $\{x_m\}$ and $(\mathcal{C}_i \cup \mathcal{C}_j) \setminus \{x_m\}$ strictly decreases the entropy. Iteratively, this shows that $(\mathcal{O}_i)_{i=1}^k$ achieves the minimum as claimed.
\end{proof}

Note that if the choice of Gibbs kernel $G$ is fixed, then $\KL{\widetilde{\pi}_\alpha}(Q_\alpha\ \|\ \widetilde{\Pi}_\alpha)$ depends on the choice of $P$ only through the term $\KL{\pi} (P\| \Pi)$, and neither $ \KL{\pi} (G\| \Pi)$ nor $H(\overline{\pi})$ depend on $P$.

The previous proposition gives a simple interpretation of the KL objective.
Since $\KL{\pi}(G\|\Pi)=H(\overline{\pi})$, the KL criterion favours partitions
where $\overline{\pi}$ has small entropy. This leads to the
minimising partition in Proposition \ref{prop: best O}, where the smallest-mass
states are placed into singleton blocks and the remaining larger-mass states are
grouped into one block. Intuitively, the KL objective favours concentrating mass
into fewer large blocks. From a computational viewpoint, however, such a large
block may be expensive to sample from exactly. Thus, while Proposition
\ref{prop: best O} identifies the partition favoured by the KL objective, this
partition may not be computationally feasible when the largest block is
comparable in size to the full state space.

We next show KL divergence of $A_\alpha$ from stationarity is in fact, bounded above by the convex combination of the KL divergence of $G$ and $P$ from $\Pi$. Naturally, Lemma \ref{lem: KL G = H} implies that the upper bound is also related to the Shannon entropy of $\overline{\pi}.$

\begin{prop} \label{prop: bound KL A_alpha}
    Let $P \in \mathcal{S}(\pi)$ be some fixed sampler and $G$ the Gibbs sampler induced by some partition $(\mathcal{O}_i)_{i=1}^k$. Then 
    $$\KL{\pi}(A_\alpha \| \Pi) \leq (1-\alpha)\KL{\pi} (G\| \Pi) + \alpha \KL{\pi}(P \| \Pi) = (1-\alpha) H(\overline{\pi}) + \alpha \KL{\pi}(P \| \Pi).$$
\end{prop}

\begin{proof}
    Note that the function $x \mapsto x\log x$ is convex. Hence, for any $x,y \in \mathcal{X}$,
    $$A_\alpha(x,y)\log A_\alpha(x,y) \leq \alpha P(x,y)\log P(x,y) + (1-\alpha)G(x,y)\log G(x,y).$$
    \begin{align*}
        \KL{\pi}(A_\alpha \| \Pi)
        &= \sum_{x,y\in\mathcal{X}}\pi(x) A_\alpha(x,y)\log\frac{A_\alpha(x,y)}{\pi(y)} \\
        &\leq \sum_{x,y\in\mathcal{X}}\pi(x) \left((1-\alpha)G(x,y)\log \frac{G(x,y)}{\pi(y)} + \alpha P(x,y)\log \frac{P(x,y)}{\pi(y)}\right)\\
        &= (1-\alpha)\KL{\pi} (G\| \Pi) + \alpha \KL{\pi}(P \| \Pi)
    \end{align*}
    The last equality follows directly from Lemma \ref{lem: KL G = H}.
\end{proof}

As with the Frobenius objective, our analysis of the KL objective has shown that $A_\alpha$ need not always outperform $P$. We re-emphasises that the choice of partition is central to the performance of $A_\alpha$, in this case through the entropy term $H(\overline{\pi})$.

Finally, we present a corollary that directly follows Proposition \ref{prop: bound KL A_alpha}.
\begin{corollary}
    Suppose $P \in \mathcal{S}(\pi)$. Then 
    \small
    $$\min_{\mathcal{O}_1,\mathcal{O}_k \neq \emptyset, \mathcal{X}}\KL{\pi}(A_\alpha \| \Pi) \leq \alpha \KL{\pi}(P \| \Pi) 
    - (1-\alpha)\left(\sum_{i=1}^{k-1} \pi(i)\log \pi(i)  + \left(\sum_{j=k}^n \pi(j)\right) \log \sum_{j=k}^n \pi(j) \right).$$
    \normalsize
\end{corollary}

\begin{proof}
    It is straightforward to see that $\min_{\mathcal{O}_1,\ldots,\mathcal{O}_k}\KL{\pi}(A_\alpha \| \Pi)$ must be smaller than all other choices of $k$-partition. 

    The right-hand side of the inequality can then be obtained by solving for $H(\overline{\pi})$, with $\overline{\pi}$ induced by the orbit in \eqref{eq: best O}.
\end{proof}

\section{Numerical experiments} \label{sec: exp}

The numerical experiments in this section are intended to illustrate three aspects of the additive sampler $A_\alpha$. First, since $A_\alpha$ combines the baseline kernel $P$ with the Gibbs kernel $G$, it is natural to compare its performance with other samplers involving the same Gibbs component, such as $GP$ and $GPG$. This places the additive construction in the context of existing group-averaging samplers. Second, we compare fixed and optimised choices of the partition in order to illustrate the partition-selection criteria studied earlier. Third, the parameter $\alpha$ is a user-specified tuning parameter which directly controls the balance between local exploration through $P$ and within-block averaging through $G$. We therefore study how the convergence behaviour changes as $\alpha$ varies.

We will be using the Curie-Weiss model as the benchmark throughout our numerical experiments in this paper. All results can be reproduced with the codes given in the repository \url{https://github.com/ryan-2357/-additive}.

Before presenting the results, we shall first recall the model and its setup. Let $\mathcal{X} = \{-1, +1\}^d$ be a $d$-dimensional state space and define the Hamiltonian function for $x = (x^1, \dots, x^d) \in \mathcal{X}$ to be  
$$\mathcal{H}(x) = - \sum_{i,j=1}^d \frac{1}{2^{|j-i|}} x^i x^j - h \sum_{i=1}^d x^i.$$
In this context, the interaction coefficient is given by $\frac{1}{2^{|j-i|}}$ and the external magnetic field is $h \in \mathbb{R}$. An in-depth discussion of the model can be found in \cite{Bovier_denHollander_2015}.

The stationary distribution we shall consider will then be the Gibbs distribution at temperature $T > 0.$ That is, 
$$\pi(x) = \frac{e^{-\frac{1}{T} \mathcal{H}(x)}}{\sum_{z\in \mathcal{X}}e^{-\frac{1}{T} \mathcal{H}(z)}}.$$

Define the Glauber dynamics with a simple random walk as 
$$P(x,y) = \begin{cases}
    \frac{1}{d} e^{-\frac{1}{T}(\mathcal{H}(y)-\mathcal{H}(x))_+}, & \mathrm{for}\ y = (x^1, \dots, -x^i, \dots, x^d), i \in \llbracket d \rrbracket,\\
    1-\sum_{y \neq x} P(x,y), & \mathrm{if}\ x = y,\\
    0, & \mathrm{otherwise,}
    \end{cases}$$
where for $m \in \mathbb{R},$ $m_+ = \max\{m,0\}$. This sampler is equivalent to uniformly picking one of the $d$ coordinates, flipping it to the opposite sign and performing an acceptance-rejection stage.
This sampler will be the baseline sampler $P$ used in all subsequent experiments. 

We also give a brief introduction to total variation distance, which will be the core metric used to discuss the performance of our samplers. 

For any two probability distributions $\mu, \nu$ on $\mathcal{X}$, we define the total variation (TV) distance between them to be 
$$\| \mu - \nu \|_{TV} := \frac{1}{2} \sum_{x\in \mathcal{X}} |\mu(x)-\nu(x)|,$$
and we denote
$$\mathbb{N} \ni l \mapsto \max_{x\in \mathcal{X}} \|P^l(x,\cdot) - \pi \|_{TV}$$
to be the worst-case TV distance of the sampler $P$.

\subsection{Comparison of $A_\alpha$ with other group-averaging samplers under fixed partition} \label{subsec: fixedS}
We compare the mixing behaviour of four kernels: the baseline Glauber dynamics $P$, the multiplicative group-averaged kernels $G_S P$ and $G_S P G_S$, and the additive kernel $A(S) = 1/2(P + G_S)$, using a fixed partition induced by the sign of the magnetisation. Specifically, we define
$$m(x) = \sum_{i=1}^d x^i$$
and set 
$$S = \{x \in \mathcal{X}: m(x) \geq 0\}.$$

Here, $S$ is chosen to separate the negative and non-negative magnetisation phases. We note that this choice is not based on a group action, but rather is intended as a simple, physically interpretable, non-optimised partition. The goal here is twofold: first, to illustrate that the proposed kernels can improve upon the baseline sampler $P$ even under a basic choice of $S$ that is not necessarily symmetry-based; and second, to provide a fixed benchmark against which the optimised choice of $S$ can later be compared.

The Gibbs kernel $G_S$ is then defined with respect to the partition $\mathcal{X} = S \sqcup S^c$.

We consider four model settings, corresponding to combinations of high and low temperature regimes ($T = 15$ and $T = 2$) and external field strengths ($h = 0$ and $h = 2$).

Figure \ref{fig:4sampler} presents the worst-case total variation (TV) distance as a function of time for all four kernels across these settings.

\begin{figure}[h]
\centering
\includegraphics[width=1\linewidth]{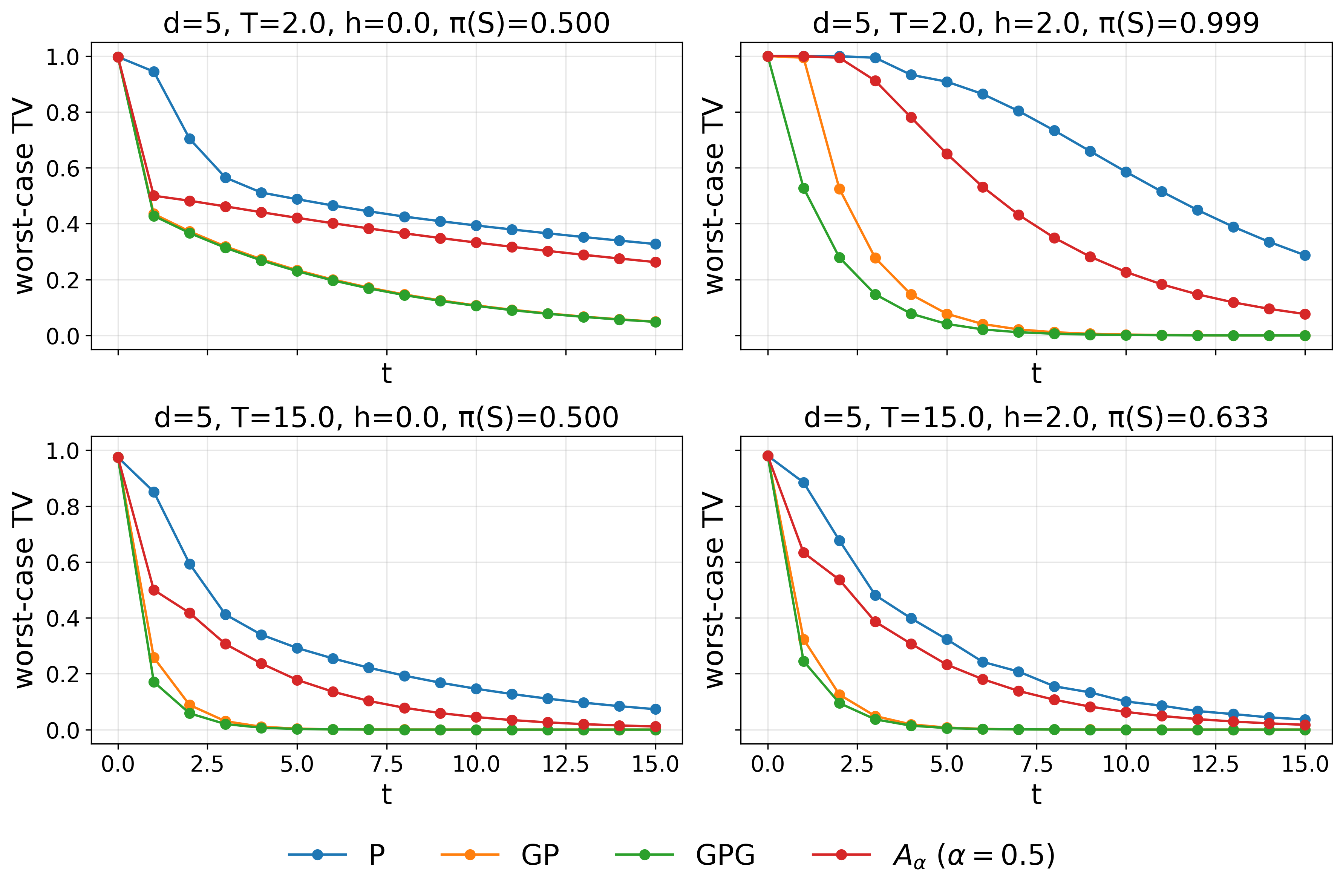}
\caption{Worst-case total variation distance for different samplers.}
\label{fig:4sampler}
\end{figure}

Across all parameter regimes, we observe a clear and consistent hierarchy in performance (from fastest to slowest mixing):
$$G_SPG_S,\ G_SP,\ A(S),\ P.$$

The composition-based kernels $G_S P G_S$ and $G_S P$ significantly accelerate convergence relative to the baseline $P$, with $G_S P G_S$ providing the strongest improvement. While the additive kernel $A(S)$ also consistently improves upon $P$, it remains slower than the composition-based alternatives.

This highlights an important structural distinction between the samplers. The kernel $A(S)$ blends local and global updates via randomisation, whereas $G_S P$ and $G_S P G_S$ combine these operations sequentially. Empirically, this sequential composition is observed to have improved mixing behaviour compared to additive samplers. 

These trends persist across both high- and low-temperature regimes, as well as across varying external fields. In particular, even in highly imbalanced settings where most of the stationary mass concentrates on one side of the partition, the composition-based kernels remain effective, whereas the baseline $P$ mixes significantly slower.

Finally, we note that the computational cost per iteration differs across the kernels. In particular, $G_S P G_S$ involves two applications of the Gibbs kernel $G_S$ and is therefore likely to be more expensive per step than $P$ or $A(S)$. As such, the present comparison evaluates convergence per kernel application, and a cost-normalised comparison may be required for a more comprehensive assessment.

\subsection{Comparison of $A_\alpha$ with other group-averaging samplers under optimal partition}

We now compare the kernels $P, G_SP, G_SPG_S$ and $A(S)$ when each sampler other than $P$ is paired with the optimal choice of $S$ under the squared Frobenius distance to stationarity. That is, we consider the partition $\mathcal{X} = S \sqcup S'$ under the objectives 
$$\Tr(PG_SP),\ \Tr(G_SPG_SPG_S),\ \frac{1}{4}\Tr(P^2 + PG_S + G_SP + G_S)$$
corresponding to the samplers $G_SP, G_SPG_S$ and $A(S)$. We perform this optimisation via brute-force search over all non-trivial sets $S$ for each of the three partition-dependent samplers. 

Figure \ref{fig:tv_optcuts} shows the worst-case TV distance of the resulting kernels paired with their optimal cuts. Figure \ref{fig:cut_by_magnetisationGP} further reveals the magnetisation profiles of the corresponding optimal cuts. 

\begin{figure}[H]
\centering
\includegraphics[width=1\linewidth]{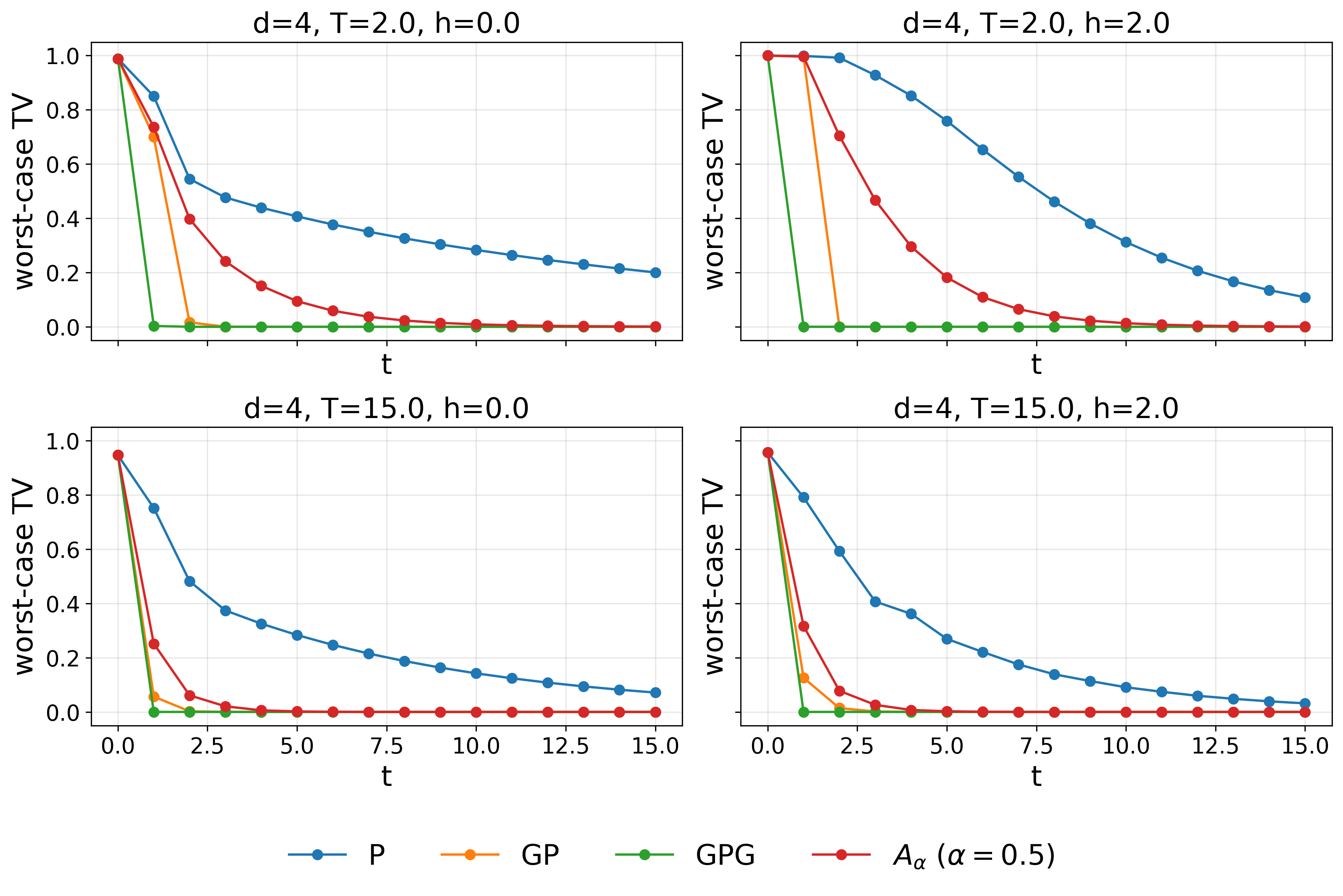}
\caption{Worst-case total variation distance for different samplers, with each partition-dependent sampler using its own Frobenius-optimal partition.}
\label{fig:tv_optcuts}
\end{figure}

\begin{figure}[H]
    \centering

    \begin{subfigure}{\linewidth}
        \centering
        \includegraphics[height=0.20\textheight]{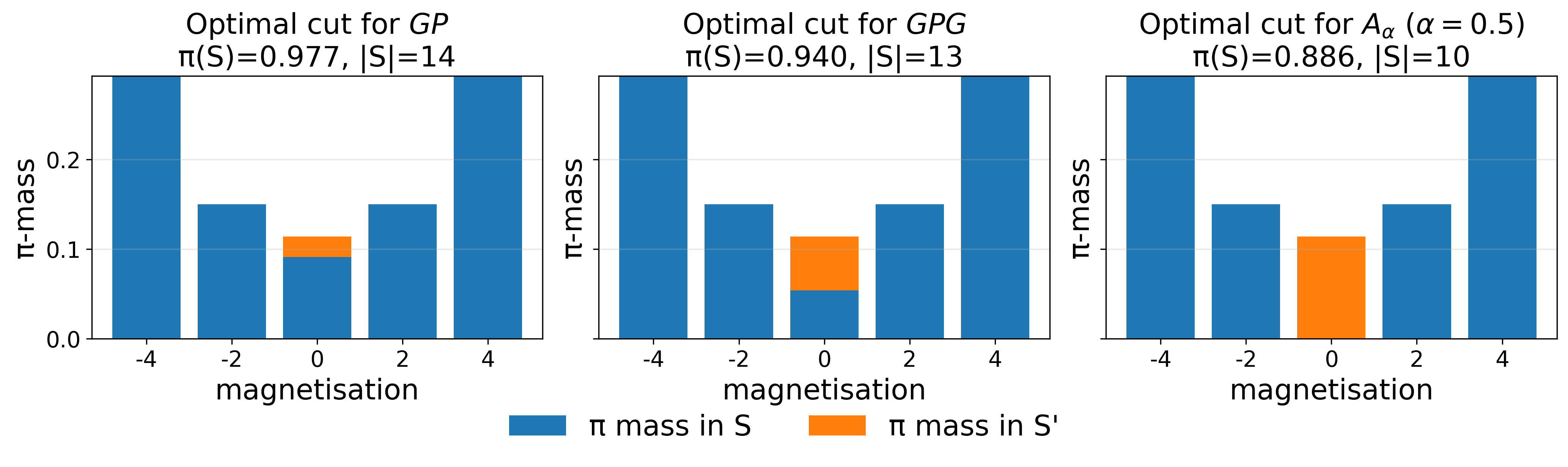}
        \caption{$T=2,\ h=0$}
    \end{subfigure}

    \begin{subfigure}{\linewidth}
        \centering
        \includegraphics[height=0.20\textheight]{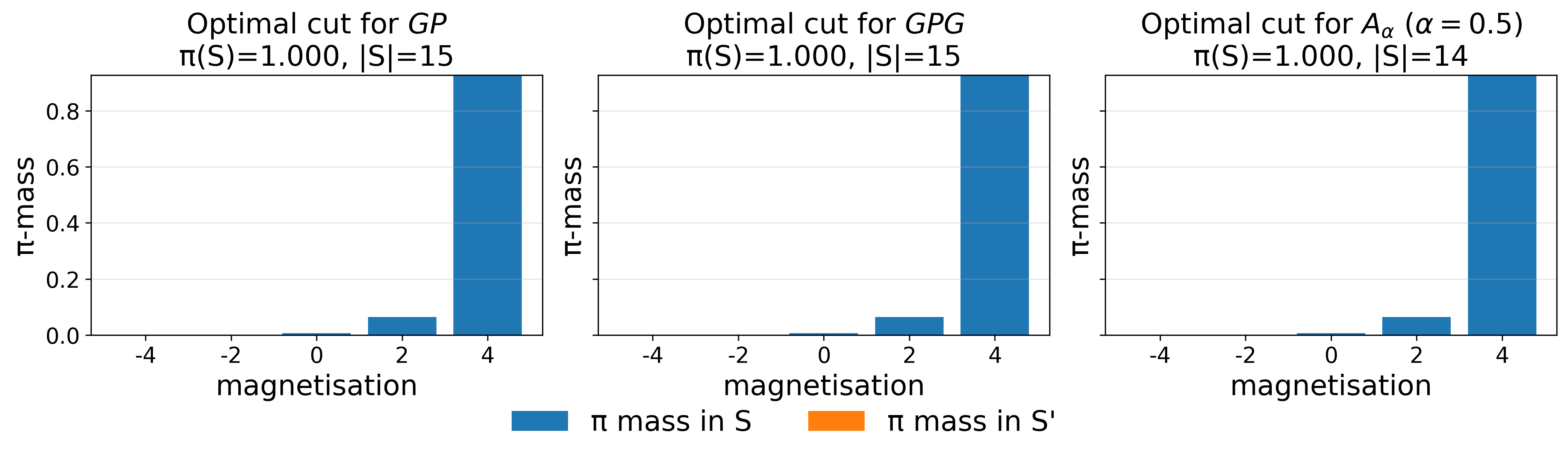}
        \caption{$T=2,\ h=2$}
    \end{subfigure}

    \begin{subfigure}{\linewidth}
        \centering
        \includegraphics[height=0.20\textheight]{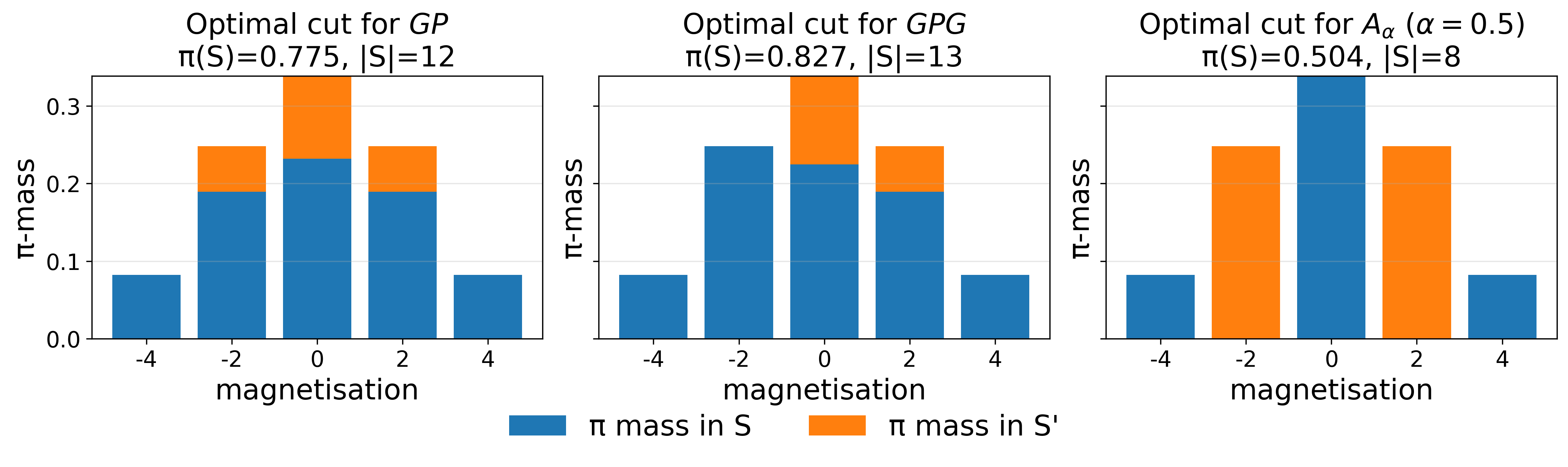}
        \caption{$T=15,\ h=0$}
    \end{subfigure}

    \begin{subfigure}{\linewidth}
        \centering
        \includegraphics[height=0.20\textheight]{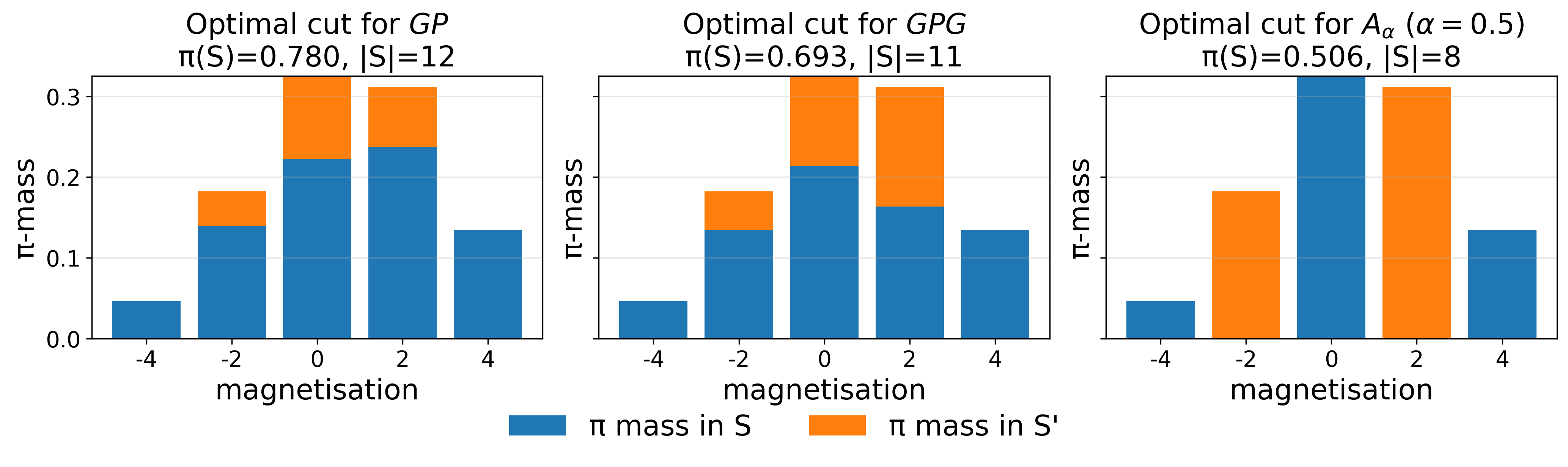}
        \caption{$T=15,\ h=2$}
    \end{subfigure}

    \caption{Magnetisation profiles of the Frobenius-optimal cuts for the Curie--Weiss model with $d=4$ for various samplers.}
    \label{fig:cut_by_magnetisationGP}
\end{figure}

Figure \ref{fig:tv_optcuts} again suggests the same hierarchy in mixing performance (from fastest to slowest):
$$G_SPG_S,\ G_SP,\ A(S),\ P.$$

Comparing both Figures \ref{fig:4sampler} and \ref{fig:tv_optcuts}, we see that all the partition-dependent samplers have improved effectiveness when using their respective optimal cuts compared to the one considered in Section \ref{subsec: fixedS}.

Figure \ref{fig:cut_by_magnetisationGP} reveals a clear structural difference between the optimal cuts selected by the three partition-dependent kernels. The Frobenius-optimal partitions for $G_SP$ and $G_SPG_S$ tend to be highly unbalanced, often placing almost all of the stationary mass into one block. This is especially pronounced in the low-temperature $(T = 2)$ regimes.

In contrast, the optimal cut for $A(S)$ are consistently more balanced in most regimes, particularly the ones with high temperature $(T=15)$. In that setting, the optimal choice of $S$ is almost symmetric with $\pi(S) \approx 1/2.$ In contrast, the optimal choice for $G_SP$ and $G_SPG_S$, while not as extreme as that in the $T=2$ case, still place noticeably more stationary mass in one block. 

Overall, the results seem to suggest that the Frobenius objective selects relatively similar cuts for the multiplicative group-averaged samplers, as compared to the additive mixture $A(S)$. For the composition-based kernels $G_SP$ and $G_SPG_S$, the optimal cuts are often highly concentrated and appear to exploit the dominant stationary mode as aggressively as possible. For the additive kernel $A(S)$, the optimal cuts are more balanced and correspond to a more moderate partition of the state space.

Finally, we again note that the comparison in Figure \ref{fig:tv_optcuts} is made per kernel application. Since one step of $G_SP$ and $G_SPG_S$ involves two applications of the Gibbs kernel, while one step of $A(S)$ involves only a single randomised update, the computational cost per iteration is not identical across samplers. Thus, the present experiment compares the kernels at the operator level rather than under a cost-normalised runtime model.

\subsection{Effect of the parameter $\alpha$ on the mixing behaviour of $A_\alpha$}

We now investigate the effect of the parameter $\alpha$ on the mixing behaviour of the additive kernel $A_\alpha(S) = \alpha P + (1-\alpha)G_S,$ under the same partition $S = \{x \in \mathcal{X} : m(x) \geq 0\}$. We first examine the full convergence trajectories for several fixed values of $\alpha$, before studying the dependence on $\alpha$ at fixed time horizons.

Figure \ref{fig:alpha_time} shows the worst-case total variation (TV) distance as a function of time for representative values of $\alpha$ given by $\{0,\ 0.25,\ 0.5,\ 0.75,\ 1\}$. Again, we consider the four regimes comprising of combinations between $T\in \{2,15\}$ and $h \in \{0, 2\}.$

\begin{figure}[h]
\centering
\includegraphics[width=1\linewidth]{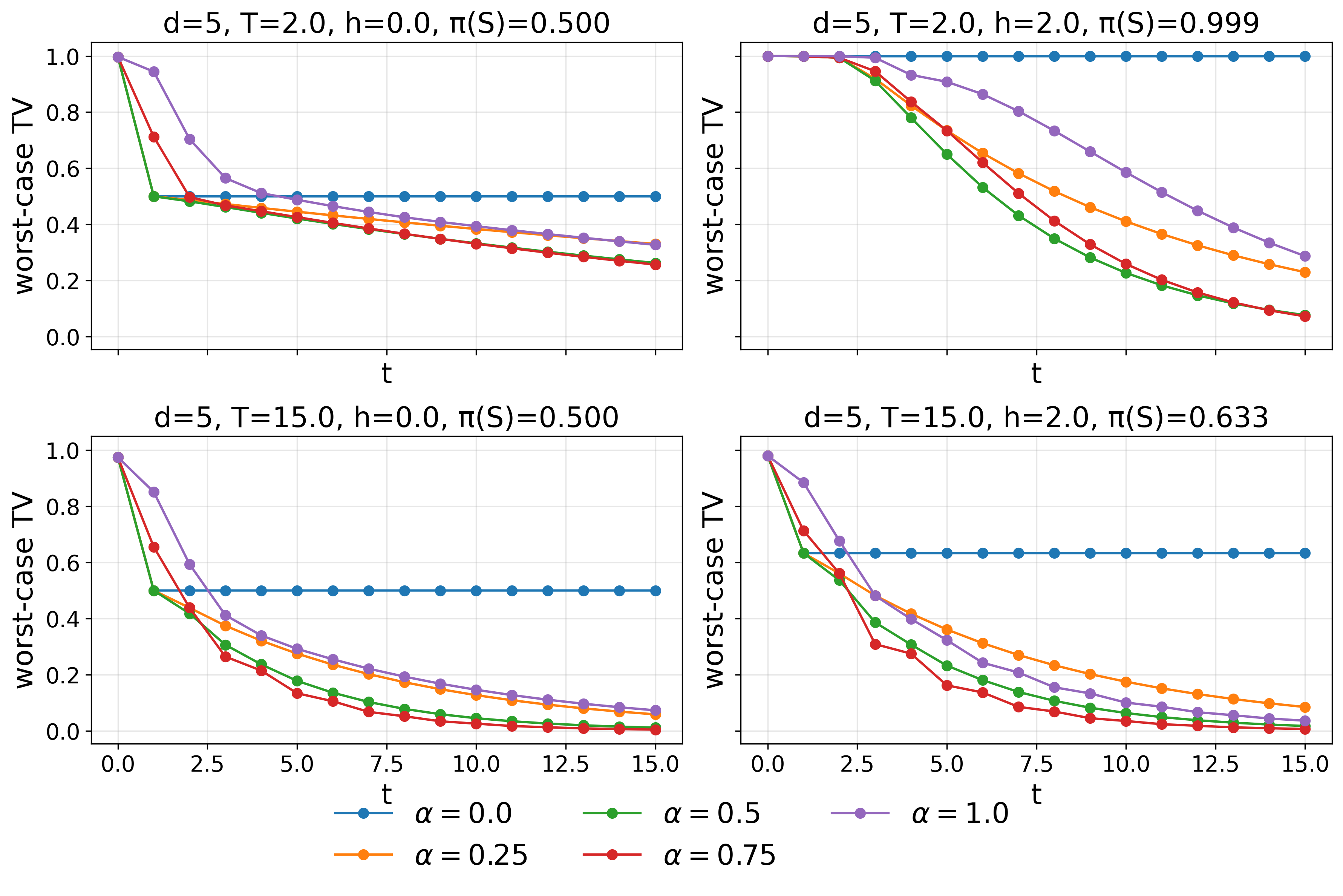}
\caption{Worst-case total variation distance of $A_\alpha(S)$ for different values of $\alpha$ against $t$.}
\label{fig:alpha_time}
\end{figure}

We observe that the extreme cases $\alpha = 0$ and $\alpha = 1$ exhibit fundamentally different limitations. When $\alpha = 0$, the kernel reduces to the Gibbs kernel $G_S$, which performs averaging within each block but does not allow transitions between blocks. As a result, the TV distance fails to converge to zero, reflecting the inability of the chain to explore the full state space. In contrast, when $\alpha = 1$, the kernel reduces to the baseline $P$, which mixes slowly due to its local update structure.

In comparison, intermediate values such as $\alpha = 0.5$ and $\alpha = 0.75$ achieve significantly faster decay of the TV distance across all regimes. This demonstrates that effective mixing requires both cross-block exploration and within-block averaging, and that $A_\alpha$ successfully combines these two mechanisms.

However, when $\alpha$ is too small, such as $\alpha = 0.25$, performance deteriorates. In some regimes, notably when $T = 15$ and $h = 2$, the TV distance is worse than that of the baseline $P$. This indicates that although $G_S$ facilitates within-block mixing, it is by itself a poor sampler due to its inability to move between blocks. Consequently, when $\alpha$ is too small, the chain spends most of its time performing within-block updates, effectively becoming trapped within orbits and leading to poor overall mixing performance.

We next quantify the dependence of the mixing behaviour on $\alpha$. Figure \ref{fig:alpha} plots the worst-case TV distance at fixed time horizons $t = 3, 5,$ and $10$ as a function of $\alpha$.

\begin{figure}[h]
    \centering
    \includegraphics[width=1\linewidth]{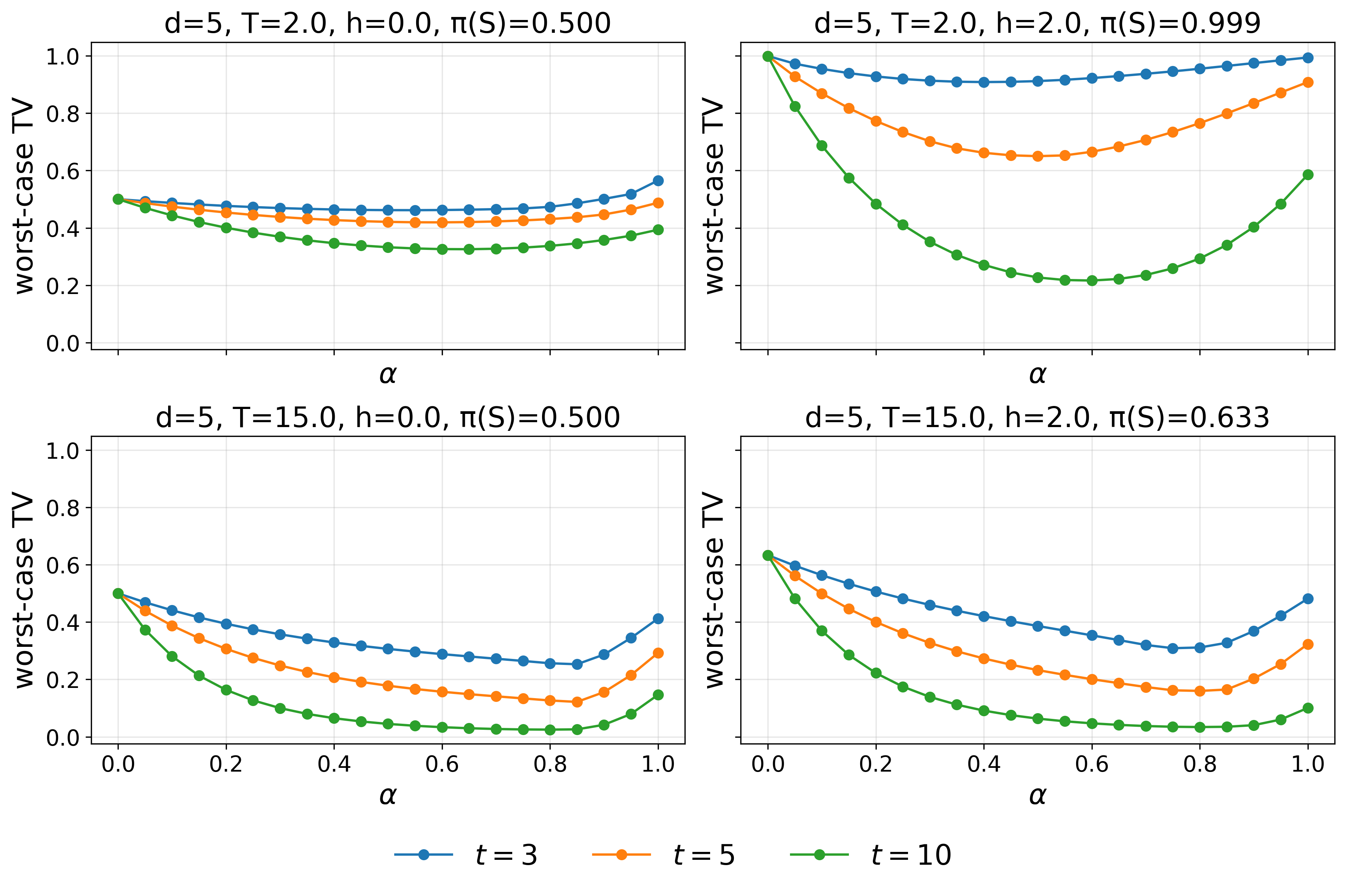}
    \caption{Worst-case total variation distance of $A_\alpha(S)$ as a function of $\alpha$ for different time horizons.}
    \label{fig:alpha}
\end{figure}

Across all regimes and time horizons, we observe a clear U-shaped dependence of the TV distance on $\alpha$. In particular, both extreme choices $\alpha = 0$ and $\alpha = 1$ lead to poorer performance, while intermediate values of $\alpha$ achieve significantly lower TV distance.

This behaviour reflects the competing roles of the two components in $A_\alpha(S)$. Pure averaging ($\alpha = 0$) fails to move between blocks, while pure local updates ($\alpha = 1$) mix slowly due to metastability. Intermediate values of $\alpha$ balance these effects, leading to improved convergence. The optimal value of $\alpha$ depends on the model parameters and the time horizon, but is consistently observed to lie close to, or above $\alpha = 0.5$

\section*{Declarations}

\paragraph{Competing interests.}
Both authors have no relevant financial or non-financial interests to disclose.

\paragraph{Data availability.}
No data was used for the research described in the article.

\section*{Acknowledgements}

We acknowledge the careful reading and constructive comments of two reviewers that have significantly improved the quality of the manuscript. Michael Choi acknowledges the financial support of the projects A-0000178-02-00 and A-8003574-00-00 at National University of Singapore.

\bibliography{ref}

@misc{GPG,
      title={{Group-averaged Markov chains II: tuning of group action in finite state space}}, 
      author={Michael C. H. Choi and Ryan J. Y. Lim and Youjia Wang},
      year={2025},
      eprint={2512.13067},
      archivePrefix={arXiv},
      primaryClass={math.PR},
      url={https://arxiv.org/abs/2512.13067}, 
}

@article{Jerrum_2004,
author={Jerrum,Mark and Son,Jung-Bae and Tetali,Prasad and Vigoda,Eric},
year={2004},
title={{Elementary Bounds on Poincaré and Log-Sobolev Constants for Decomposable Markov Chains}},
journal={The Annals of applied probability},
volume={14},
number={4},
pages={1741-1765},
isbn={1050-5164},
language={English},
}

@misc{submodular,
      title={Optimising two-block averaging kernels to speed up Markov chains}, 
      author={Ryan J. Y. Lim and Michael C. H. Choi},
      year={2026},
      eprint={2603.10318},
      archivePrefix={arXiv},
      primaryClass={math.PR},
      url={https://arxiv.org/abs/2603.10318}, 
}

@article{Montenegro,
author = {Ravi Montenegro},
title = {{Sharp edge, vertex, and mixed Cheeger type inequalities for finite Markov kernels}},
volume = {12},
journal = {Electronic Communications in Probability},
publisher = {Institute of Mathematical Statistics and Bernoulli Society},
pages = {377 -- 389},
year = {2007},
doi = {10.1214/ECP.v12-1269},
URL = {https://doi.org/10.1214/ECP.v12-1269}
}

@book{roch_mdp_2024,
place={Cambridge},
series={Cambridge Series in Statistical and Probabilistic Mathematics},
title={{Modern Discrete Probability: An Essential Toolkit}},
DOI={10.1017/9781009305129},
publisher={Cambridge University Press},
author={Roch, Sebastien},
year={2024},
collection={Cambridge Series in Statistical and Probabilistic Mathematics}
}

@misc{choi2026geometryfactorizationmultivariatemarkov,
      title={Geometry and factorization of multivariate Markov chains with applications to MCMC acceleration and approximate inference}, 
      author={Michael C. H. Choi and Youjia Wang and Geoffrey Wolfer},
      year={2026},
      eprint={2404.12589},
      archivePrefix={arXiv},
      primaryClass={math.PR},
      url={https://arxiv.org/abs/2404.12589}, 
}

@inbook{Krause_Golovin_2014, place={Cambridge}, title={{Submodular Function Maximization}}, booktitle={{Tractability: Practical Approaches to Hard Problems}}, publisher={Cambridge University Press}, author={Krause, Andreas and Golovin, Daniel}, year={2014}, pages={71–104}}

@article{uriel_2011,
author = {Feige, Uriel and Mirrokni, Vahab S. and Vondr\'{a}k, Jan},
title = {Maximizing Non-monotone Submodular Functions},
journal = {SIAM Journal on Computing},
volume = {40},
number = {4},
pages = {1133-1153},
year = {2011},
doi = {10.1137/090779346},
URL = {https://doi.org/10.1137/090779346
},
eprint = { 
        https://doi.org/10.1137/090779346
}}

@misc{iyer2013algorithmsapproximateminimizationdifference,
      title={{Algorithms for Approximate Minimization of the Difference Between Submodular Functions, with Applications}}, 
      author={Rishabh Iyer and Jeff Bilmes},
      year={2013},
      eprint={1207.0560},
      archivePrefix={arXiv},
      primaryClass={cs.DS},
      url={https://arxiv.org/abs/1207.0560}, 
}

@book{Bovier_denHollander_2015,
    author={Bovier,Anton and den Hollander,Frank},
    year={2015},
    title={{Metastability: A Potential-Theoretic Approach}},
    publisher={Springer International Publishing},
    address={Cham},
    volume={351},
    edition={1st 2015.;1;},
    isbn={0072-7830},
    language={English},
}

@inproceedings{chen_lifting,
author = {Chen, Fang and Lov\'{a}sz, L\'{a}szl\'{o} and Pak, Igor},
title = {Lifting Markov chains to speed up mixing},
year = {1999},
isbn = {1581130678},
publisher = {Association for Computing Machinery},
address = {New York, NY, USA},
url = {https://doi.org/10.1145/301250.301315},
doi = {10.1145/301250.301315},
booktitle = {Proceedings of the Thirty-First Annual ACM Symposium on Theory of Computing},
pages = {275–281},
numpages = {7},
location = {Atlanta, Georgia, USA},
series = {STOC '99}
}

@article{Vucelja_lifting,

author={Vucelja,Marija},

year={2016},

title={Lifting—A nonreversible Markov chain Monte Carlo algorithm},

journal={American journal of physics},

volume={84},

number={12},

pages={958-968},

isbn={0002-9505},

language={English},

}

@article{TURITSYN2011410,
title = {Irreversible Monte Carlo algorithms for efficient sampling},
journal = {Physica D: Nonlinear Phenomena},
volume = {240},
number = {4},
pages = {410-414},
year = {2011},
issn = {0167-2789},
doi = {https://doi.org/10.1016/j.physd.2010.10.003},
url = {https://www.sciencedirect.com/science/article/pii/S0167278910002782},
author = {Konstantin S. Turitsyn and Michael Chertkov and Marija Vucelja}
}

@article{levine_gibbs,

author={Levine,Richard A. and Casella,George},

year={2008},

title={Comment: On Random Scan Gibbs Samplers},

journal={Statistical science},

volume={23},

number={2},

pages={192-195},

isbn={0883-4237},

language={English},

}

@article{roberts_geometric,
author = {Gareth Roberts and Jeffrey Rosenthal},
title = {{Geometric Ergodicity and Hybrid Markov Chains}},
volume = {2},
journal = {Electronic Communications in Probability},
number = {none},
publisher = {Institute of Mathematical Statistics and Bernoulli Society},
pages = {13 -- 25},
year = {1997},
doi = {10.1214/ECP.v2-981},
URL = {https://doi.org/10.1214/ECP.v2-981}
}

@misc{choi2025groupaveragedmarkovchainsmixing,
      title={Group-averaged Markov chains: mixing improvement}, 
      author={Michael C. H. Choi and Youjia Wang},
      year={2025},
      eprint={2509.02996},
      archivePrefix={arXiv},
      primaryClass={math.PR},
      url={https://arxiv.org/abs/2509.02996}, 
}
\end{document}